\documentclass[a4paper,10pt,leqno]{amsart}
\usepackage[cp850]{inputenc}

\usepackage[dvips]{graphics}

\usepackage{psfrag}

\usepackage{amsxtra,amssymb}

\usepackage{mathrsfs}

\numberwithin{equation}{section}

\newtheorem{theo}{Theorem}[section]

\newtheorem{prop}[theo]{Proposition}

\newtheorem{lemm}[theo]{Lemma}

\newtheorem{coro}[theo]{Corollary}

\theoremstyle{definition}

\theoremstyle{remark}

\begin{document}

\baselineskip=10pt

\title{Diophantine Approximation and special Liouvillenumbers}

\author{ Johannes Schleischitz}
                            
\maketitle

\vspace{8mm}

\begin{abstract}
This paper introduces some methods to determine the simultaneous approximation constants 
of a class of well approximable numbers $\zeta_{1},\zeta_{2},\ldots ,\zeta_{k}$.
The approach relies on results on the connection between the set of all $s$-adic expansions ($s\geq 2$)
of $\zeta_{1},\zeta_{2},\ldots ,\zeta_{k}$ and their associated approximation constants.
As an application, explicit construction of real numbers $\zeta_{1},\zeta_{2},\ldots ,\zeta_{k}$ with
prescribed approximation properties are deduced and illustrated by Matlab plots.
\end{abstract}
{\footnotesize{ AMS 2010 Mathematics Subject  Classification: 11J13, 11H06, 11J81}}\\
 \vspace{4mm}
 \indent {\footnotesize{ Supported by FWF grant P22794-N13}}

\vspace{8mm}

\section{Introduction}

\maketitle

\subsection{Basic facts and notations.} 
This paper deals with the one parameter simultaneous approximation problem

\begin{eqnarray}
 \vert x\vert &\leq& Q^{1+\theta} \label{eq:rof}  \\
 \vert \zeta_{1}x-y_{1}\vert &\leq& Q^{-\frac{1}{k}+\theta}  \nonumber  \\
 \vdots \quad &\vdots& \quad \vdots   \nonumber \\
 \vert \zeta_{k}x-y_{k}\vert &\leq& Q^{-\frac{1}{k}+\theta},   \nonumber
\end{eqnarray}
where $\zeta_{1},\zeta_{2},\ldots ,\zeta_{k}$ are real numbers which we will assume to be linearly independent together with $1$
 and $x,y_{1},y_{2},\ldots ,y_{k}$ are integers 
to be determined in dependence of the parameter $Q>1$ in order to minimize $\theta$. To be more precise, we 
define the function $\psi_{j}(Q)$ 
for $1\leq j\leq k+1$ by setting $\psi_{j}(Q)$ the minimum over all $\theta\in{\mathbb{R}}$ such that
there are $j$ linearly independent vectors $(x,y_{1},y_{2},\ldots ,y_{k})\in{\mathbb{Z}^{k+1}}$ that satisfy the system 
(\ref{eq:rof}). In the sequel we will restrict to approximation vectors with $x>0$, which clearly is no loss of generality as
$(x,y_{1},\ldots ,y_{k})\mapsto (-x,-y_{1},\ldots ,-y_{k})$ does not affect approximation constants.
Another equivalent way to view the functions $\psi_{j}$ is to consider the lattice
$\Lambda=\{(x,\zeta_{1}x-y_{1},\ldots ,\zeta_{k}x-y_{k}):x,y_{1},\ldots y_{k}\in{\mathbb{Z}} \}$
and the convex body (in fact the parallelepiped) $K(Q)$ defined as the set of points 
$(z_{1},z_{2},\ldots ,z_{k+1})\in{\mathbb{R}^{k+1}}$ with

\begin{eqnarray}
  \vert z_{1}\vert &\leq& Q    \label{eq:tommot}  \\
  \vert z_{i}\vert &\leq& Q^{-\frac{1}{k}},\qquad 2\leq i\leq k+1,   \label{eq:tomot}
\end{eqnarray}
and to define $\lambda_{j}(Q)$ as the $j$-th {\em successive minimum} of $\Lambda$ with respect to $K(Q)$.
This $j$-th minimum is defined as the infimum over all $\lambda>0$ for which the $\mathbb{R}$-span of $\lambda K(Q)\cap \Lambda$ 
has dimension at least $j$, or equivalently $\lambda K(Q)$ contains $j$ linearly independent points of $\Lambda$. \\
With respect to these successive minima $\lambda_{j}$, the functions $\psi_{j}(Q)$ can also be determined by

\[
 Q^{\psi_{j}(Q)}=\lambda_{j}(Q).
\]
One has the inequalities
 
 \begin{equation} \label{eq:mbni}
 -1\leq \psi_{1}(Q)\leq \psi_{2}(Q)\leq \ldots \leq \psi_{k+1}(Q)\leq \frac{1}{k} 
\end{equation}
as we will show later, and Dirichlet's Theorem states

\begin{equation} \label{eq:11}
 \psi_{1}(Q)<0 \qquad \text{for all} \quad Q>1.
\end{equation}
Minkowski's second convex body theorem yields for any convex body $K$ with volume $V(K)$ and any lattice $\Lambda$

\[
 \frac{2^{k+1}}{(k+1)!}\frac{\det(\Lambda)}{V(K)}\leq \lambda_{1}\lambda_{2}\cdots \lambda_{k+1}\leq 2^{k+1}\frac{\det(\Lambda)}{V(K)},
\]
see \cite{1}, so that in our special case, as $V(K(Q))=1$ for every $Q$, we have                                                               

\[
 c_{1}(\Lambda)\leq \lambda_{1}(Q)\lambda_{2}(Q)\cdots \lambda_{k+1}(Q) \leq c_{2}(\Lambda)
\]
uniformly in the parameter $Q$. With $q:=\log(Q)$ and taking logarithms, this yields

\begin{equation} \label{eq:zokz}
 q\left\vert \sum_{i=1}^{k+1} \psi_{i}(Q)\right\vert \leq C(\Lambda),
\end{equation}
with some constant $C(\Lambda)$ not depending on $Q$.\\
Another important property of the joint behaviour of the functions $\psi_{j}$ is that for any given $1\leq s\leq k$ there 
are arbitrarily large values $Q=Q(s)$ such that 

\begin{equation}  \label{eq:asdf}
\psi_{s}(Q)=\psi_{s+1}(Q)
\end{equation}
provided that $1,\zeta_{1},\zeta_{2},\ldots ,\zeta_{k}$ are linearly independent over $\mathbb{Q}$, see Theorem 1.1 in \cite{5}.                   
To quantify the behaviour of $\psi_{j}(Q)$ Summerer and Schmidt introduced the quantities 

\[
\underline{\psi}_{i}:= \liminf_{Q\to\infty} \psi_{i}(Q),\qquad \overline{\psi}_{i}:= \limsup_{Q\to\infty} \psi_{i}(Q),
\]
and gave the estimates

\begin{eqnarray}
  \underline{\psi}_{j}&\geq& \frac{j-k-1}{kj},\qquad 1\leq j\leq k+1  \label{eq:1000alk} \\
  \overline{\psi}_{j} &\geq& \frac{j-k}{k(j+1)},\qquad 1\leq j\leq k,  \label{eq:1001alk}
\end{eqnarray}
where (\ref{eq:1001alk}) requires $1,\zeta_{1},\zeta_{2},\ldots ,\zeta_{k}$ to be linearly independent over $\mathbb{Q}$ again.
Each of these bounds will be shown to be best possible in Corollary \ref{korol1}.
Moreover, (\ref{eq:asdf}) implies

\begin{equation} \label{eq:rmuo}
 \underline{\psi}_{i+1}\leq \overline{\psi}_{i},\qquad 1\leq i\leq k.
\end{equation} 
In order to study the dynamical behaviour of the functions $\psi_{j}(Q)$ it will be convenient to work with functions 

\[
 L_{j}(q)= q\psi_{j}(Q)
\]
as these functions are piecewise linear with slopes among $\{-1,\frac{1}{k}\}$. Therefore we have (\ref{eq:mbni}).
Defintion (\ref{eq:zokz}) is equivalent to 

\begin{equation} \label{eq:zoku}
\left\vert \sum_{i=1}^{k+1} L_{i}(q)\right\vert \leq C(\Lambda).
\end{equation}
We also introduce the classical 
approximation constants $\omega_{j},\widehat{\omega}_{j}$ defined by Jarnik, Bugeaud in addition to
$\underline{\psi}_{j}, \overline{\psi}_{j}$. For fixed $\zeta_{1},\zeta_{2},\ldots ,\zeta_{k}$
and for every $X>0$ define the functions $\omega_{j}(X)$ as the supremum over all real numbers $\nu$ (in fact the maximum)
 such that the system

\begin{equation}  \label{eq:1}
 \vert x\vert \leq X, \quad \vert \zeta_{i}x- y_{i}\vert \leq X^{-\nu}, \qquad 1\leq i \leq k,
\end{equation}
has $j$ linearly independent solutions $(x,y_{1},\ldots ,y_{k})\in{\mathbb{Z}^{k+1}}$. 
The approximation constants $\omega_{j},\widehat{\omega}_{j}$ are now defined as

\begin{equation*}
 \omega_{j}= \limsup_{X\to\infty} \omega_{j}(X), \qquad   \widehat{\omega}_{j}= \liminf_{X\to\infty} \omega_{j}(X). 
\end{equation*}
We will put $\omega:=\omega_{1}, \widehat{\omega}:=\widehat{\omega}_{1}$ and denote by 
$\Omega=(\omega,\omega_{2},\ldots ,\omega_{k+1},\widehat{\omega},\ldots ,\widehat{\omega}_{k+1})\in{\mathbb{R}^{2k+2}}$ the 
vector of classical approximation constants (relative to $\zeta_{1},\zeta_{2},\ldots ,\zeta_{k}$).
Very similar to the proof of Theorem 1.4 in \cite{5}, which                                                                                             
treats the special case $j=1$, one obtains 

\begin{equation} \label{eq:2}
 (1+\omega_{j})(1+\underline{\psi}_{j})= (1+\widehat{\omega}_{j})(1+\overline{\psi}_{j})= \frac{k+1}{k},\qquad 1\leq j\leq k+1.
\end{equation}
One just needs to replace ''a solution'' by ''$j$ linearly independent solutions'' at any place it occurs in the proof.
Combining (\ref{eq:2}) with (\ref{eq:1000alk}),(\ref{eq:1001alk}) for $1,\zeta_{1},\zeta_{2},\ldots ,\zeta_{k}$ linearly independent over $\mathbb{Q}$ we obtain the bounds  
 
\begin{eqnarray}
 \frac{1}{k}&\leq& \omega \quad \leq \quad \infty   \label{eq:2000neu}   \\
 \frac{1}{k}&\leq& \omega_{2} \quad \leq \quad 1,           \label{eq:2222neu}      \\
 0 &\leq& \omega_{j} \quad \leq \quad \frac{1}{j-1},  \qquad 3\leq j\leq k+1   \label{eq:2002neu}  
\end{eqnarray}
for the constants $\omega_{j}$ as well as

\begin{eqnarray}
  \frac{1}{k}&\leq& \widehat{\omega} \quad \leq \quad 1      \label{eq:2001neu}      \\
 0 &\leq& \widehat{\omega}_{j} \quad \leq \quad \frac{1}{j}, \qquad 2\leq j\leq k.   \label{eq:2003neu}  \\
 0 &\leq& \widehat{\omega}_{k+1} \leq \quad \frac{1}{k}                       \label{eq:2004neu}
\end{eqnarray}
for the constants $\widehat{\omega}_{j}$. Each considered individually, these bounds
 again are best possible.\\

\subsection{Outline of the results}
In the present paper, we will put our focus on simultaneous approximation of numbers that allow good individual as well as
simultaneous approximation. Liouville numbers, that is real numbers $\zeta$ for which the inequality 

 \[
 \left\vert\zeta-\frac{p}{q}\right\vert\leq \frac{1}{q^{\eta}}
\]
has infinitely many rational solutions $\frac{p}{q}$ for arbitrarily large $\eta\in{\mathbb{R}}$, will be suitable examples
since they all satisfy $\omega=\infty$, where $\omega=\omega_{1}$ is defined by (\ref{eq:1}) in the onedimensional case.\\
\noindent In section 2, Propositions \ref{propo},
\ref{prop2}, we establish a connection between the $s$-adic expansions ($s\geq 2$) of the components $\zeta_{j}$ of
$(\zeta_{1},\zeta_{2},\ldots ,\zeta_{k})$ and the approximation constants $\omega,\widehat{\omega}$. 
These results are then applied to the case where all $\zeta_{j}$ admit good approximations in one fixed base
$s$ independent of $j$. After these considerations for suitable arbitrary $(\zeta_{1},\zeta_{2},\ldots ,\zeta_{k})$ we put our
focus on Liouville numbers, using heavily the fact that $\omega=\infty$ in this case. Theorem \ref{satz1} will allow to compute
all classical approximation constants $\omega_{j},\widehat{\omega}_{j}$ for a special type of Liouville numbers and the resulting 
Corollary \ref{koroo} will lead us to the construction of vectors $(\zeta_{1},\zeta_{2},\ldots ,\zeta_{k})$
with prescribed approximation constants $\omega_{j},\widehat{\omega}_{j}$ that are subject to certain restrictions.
As consequences of these results we will be able to give an explicit example of a vector $\zeta_{1},\zeta_{2},\ldots ,\zeta_{k}$
that shows a conjecture by
Wolfgang Schmidt concerning successive minima of a lattice to be true. A non-constructive proof was given by Moshchevitin in a nonconstructive way. 
Moreover we will construct cases where all functions $\psi_{j}$ simultaneously take all possible values of their spectrum 
for arbitrarily large $Q$.\\
Inspired by methods used to deal with Liouville numbers, we then gegeneralize Theorem \ref{satz1} to a wider class of vectors
$(\zeta_{1},\zeta_{2},\ldots ,\zeta_{k})$ for which $\omega<\infty$. This will be the subject of Theorems \ref{satz2},\ref{satz6}
 and lead to many more explicit constructions of special cases of the Schmidt Conjecture. \\

\noindent In the last section we will first discuss the special case where $\underline{\psi}_{j+1}=\overline{\psi}_{j}$ for $1\leq j\leq k$
and give a constructive existence proof for the degenerate case $\underline{\psi}_{1}=-1$ in arbitrary dimension. 
Troughout the paper we will illustrate the derived results by Matlab plots of the functions $L_{j}$ for the special cases we consider
 to visualize derived results.
 These plots shall also
lead to some insight into the dynamical behaviour of these functions in general.
 One should mention at this point that the plots 
often seem curved although the functions are piecewise linear, which is due to the non sufficient digital reslution, 
i.e. by zooming in one can see that they are indeed piecewise linear.  

\section{Results for Approximation constants}

\subsection{Estimates for $\omega,\widehat{\omega}$}
In the sequel let $s\geq 2$ be an integer and $\zeta_{i}\in{(0,1)}$ for $1\leq i\leq k$.
For each $1\leq i\leq k$ the non vanishing digits of the $s$-adic expansions of such $\zeta_{i}$ and $1-\zeta_{i}$ define 
two sequences $(a_{n}^{i,(s)})_{n\geq 1}$ and $(a_{n}^{\prime i,(s)})_{n\geq 1}$ by

\begin{eqnarray}
\zeta_{i}= \sum_{n\geq 1} \alpha_{n,i}^{(s)}s^{-a_{n}^{i,(s)}}, \qquad a_{1}^{i,(s)}<a_{2}^{i,(s)}<\ldots, \quad 0< \alpha_{n,i}\leq s-1  \label{eq:ttt}\\
1-\zeta_{i}= \sum_{n\geq 1} \beta_{n,i}^{\prime (s)}s^{-a_{n}^{\prime i,(s)}}, \qquad a_{1}^{\prime (s)}<a_{2}^{\prime (s)}<\ldots, \quad 0< \beta_{n,i}^{\prime}\leq s-1. \label{eq:tttt}
\end{eqnarray}
We call the sequence $a_{n,i}^{\prime (s)}$ the {\em dual expansion} of $\zeta_{i}$ in base $s$.
Set $(b_{n}^{(s)})_{n\geq 1}$ the monotonically ordered sequence of all $(a_{n,i}^{i,(s)})_{n\geq 1}$ and similarly
$(b_{n}^{\prime (s)})_{n\geq 1}$ the monotonically ordered sequence of all $(a_{n}^{\prime i,(s)})_{n\geq 1}$.
The following Theorem expresses the simultaneous approximation constant $\omega$ of $\zeta_{1},\zeta_{2},\ldots ,\zeta_{k}$ 
in terms of the $s$-adic presentations of $\zeta_{i}$ $(s=2,3,4,\ldots )$ by using these two orderd sequences. The proof 
 is introductory to the rest of the work and for this purpose quite detailed.

\begin{prop} \label{propo}
We have
\begin{eqnarray}  
\omega &\geq& \max\left\{\limsup_{\vert\vert(s,n)\vert\vert_{\infty}\to \infty} \frac {b_{n+1}^{(s)}- b_{n}^{(s)}-1}{b_{n}^{(s)}}, \limsup_{\vert\vert(s,n)\vert\vert_{\infty}\to \infty} \frac {b_{n+1}^{\prime(s)}- b_{n}^{\prime(s)}-1}{b_{n}^{\prime(s)}}\right\}, \quad \text{and} \label{eq:3} \\
\omega &\leq&  \max\left\{\limsup_{\vert\vert(s,n)\vert\vert_{\infty}\to \infty} \frac {b_{n+1}^{(s)}- b_{n}^{(s)}}{b_{n}^{(s)}}, \limsup_{\vert\vert(s,n)\vert\vert_{\infty}\to \infty} \frac {b_{n+1}^{\prime(s)}- b_{n}^{\prime(s)}}{b_{n}^{\prime(s)}}\right\}.   \label{eq:qpqp}
\end{eqnarray}
where $\vert\vert(A,B)\vert\vert_{\infty}:=\max\{\vert A\vert,\vert B\vert\}$ (or any other norm since they are all equivalent in $\mathbb{R}^{2}$).
\end{prop}

\begin{proof}
We first prove (\ref{eq:3}). By definition of $(b_{n}^{(s)})_{n\geq 1}$ as the mixed ordered sequence of the sequences
$(a_{n})_{n\geq 1}$, all numbers $\zeta_{1},\zeta_{2},\ldots ,\zeta_{k}$ will have zeros at the positions 
$b_{n}^{(s)}+1,b_{n}^{(s)}+2,\ldots ,b_{n+1}^{(s)}-1$ behind the comma in base $s$ for any $s\geq 2$. Since multiplication 
of $\zeta_{j}$ by $s^{b_{n}^{(s)}}$ only shifts the comma $b_{n}^{(s)}$ positions to the right, this means, for any $1\leq j\leq k$
all $s^{b_{n}^{(s)}}\zeta_{j}$ start with $b_{n+1}-b_{n}-1$ zeros in base $s$ behind the comma. For this reason 
any pair $(s,n)$ satisfies

\[
 \left\vert\left\vert s^{b_{n}(s)}\zeta_{j}\right\vert\right\vert=  \left\vert s^{b_{n}^{(s)}}\zeta_{j}-\left\lfloor s^{b_{n}^{(s)}}\zeta_{j}\right\rfloor \right\vert \leq s^{-(b_{n+1}^{(s)}-b_{n}^{(s)}-1)}
\]
for any $1\leq j\leq k$, where $\vert\vert .\vert\vert$ denotes the smallest distance of a real number to an integer. Analoguously, for all $1-\zeta_{j}$ and all pairs $(s,n)$ we have

\[
\left\vert\left\vert s^{b_{n}^{\prime (s)}}\zeta_{j}\right\vert\right\vert= \left\vert\left\vert s^{b_{n}^{\prime (s)}}(1-\zeta_{j})\right\vert\right\vert\leq \left\vert \left\lceil s^{b_{n}^{\prime (s)}}\zeta_{j} \right\rceil- s^{b_{n}^{\prime (s)}}\zeta_{j} \right\vert \leq s^{-(b_{n+1}^{\prime (s)}-b_{n}^{\prime (s)}-1)}.
\]
We conclude that for any pair $(s,n)$

\begin{equation} \label{eq:rmna}
 \max_{1\leq j\leq k} \left\vert\left\vert x\zeta_{j}\right\vert\right\vert \leq \max\left\{s^{-(b_{n+1}^{(s)}-b_{n}^{(s)}-1)},s^{-(b_{n+1}^{\prime (s)}-b_{n}^{\prime (s)}-1)}\right\}, \quad \text{with} \quad x= s^{b_{n}^{(s)}}\quad \text{or} \quad x=s^{b_{n}^{\prime (s)}}.
\end{equation}
Surely, $s^{b_{n}}\to \infty$ or $s^{b_{n}^{\prime}}\to \infty$ is equivalent to
  $\vert\vert (s,n)\vert\vert_{\infty} \to\infty$, and we claim that (\ref{eq:3}) follows directly from the definition of 
the approximation constant
 $\omega$. To see this we take a sequence of pairs $(n,s)$ with $\vert\vert(s,n)\vert\vert_{\infty} \to\infty$,
 for which
 $\frac{b_{n+1}^{(s)}- b_{n}^{(s)}-1}{b_{n}^{(s)}}$ or $\frac {b_{n+1}^{\prime(s)}- b_{n}^{\prime(s)}-1}{b_{n}^{\prime(s)}}$ 
tend to the $\limsup$-values on the right hand side of (\ref{eq:3}).
 Putting $X_{\sigma(n,s)}:=x_{\sigma(n,s)}:=s^{b_{n}^{(s)}}$ or $X^{\prime}_{\sigma(n,s)}:=x_{\sigma(n,s)}:=s^{b_{n}^{\prime (s)}}$, 
where $\sigma$ is an arbitrary  
bijection $\mathbb{N}\times \mathbb{N}\to \mathbb{N}$, we obtain a sequence of $X$-values and $x$-values that leads via (\ref{eq:rmna})
to an approximation constant $\omega$ in (\ref{eq:1})
at least as large as both $\limsup$-values. \\

\noindent To prove (\ref{eq:qpqp}), we first show the following assertion:
It suffices to prove, that for any sufficiently large real parameter $X$ there is a $s_{0}=s_{0}(X)$, such that

\begin{equation} \label{eq:anelim}
 \frac{b_{2}^{(s_{0})}-b_{1}^{(s_{0})}}{b_{1}^{(s_{0})}}\geq \nu \quad \text{or} \quad \frac{b_{2}^{\prime (s_{0})}-b_{1}^{\prime (s_{0})}}{b_{1}^{\prime (s_{0})}}\geq \nu,
\end{equation}
where $\nu=\nu(s)$ is the largest exponent for which 

\begin{equation} \label{eq:merz}
 \max_{1\leq j\leq k}\vert\vert \zeta_{j}s\vert\vert = s^{-\nu}
\end{equation}
holds for all $s\leq X$. \\
\noindent For any sequence $(X_{i})_{i\geq 1}$ 
let $(\nu_{i})_{i\geq 1}$ be the largest exponent, for which (\ref{eq:merz}) holds with $\nu_{i}$ in place of $\nu$ 
for all $s\leq X_{i}$.
The existence of $s_{0}=s_{0}(X)$ with (\ref{eq:anelim}) for any $X$ implies the existence of a sequence $(\beta_{i})_{i\geq 1}$ with
$\beta_{i}\geq \nu_{i}$ for all $i$ with $\beta_{i}$ of the shape $\frac{b_{2}^{(s_{0})}-b_{1}^{(s_{0})}}{b_{1}^{(s_{0})}}$
hence of the shape of the expressions involved in (\ref{eq:qpqp}) in the case $n=1$.
 By definition of $\omega$ we may choose the sequence $(X_{i})_{i\geq 1}$ such
that $\lim_{i\to\infty}\nu_{i}=\limsup_{i\to\infty}\nu_{i}=\omega$. Furthermore we can assume without loss of generality
that $(X_{i})_{i\geq 1}$ satisfies $s_{i}=X_{i}$ for any $i$,
as the exponent $\nu$ in the definition of $\omega$ in (\ref{eq:1}) for a fixed $x$ decreases with growing $X$.
 Combining all these observations we get
$\limsup_{i\to\infty} \beta_{i}\geq \omega$ where $\beta_{i}$ fits in the $\limsup$ term
of (\ref{eq:qpqp}) if we set $(s_{i},n_{i})=(s_{i},1)$, where $s_{i}$ plays the role of $s_{0}$ above, for $X=X_{i}$.\\ 
\noindent It remains to prove that for such sequences we have 
$\lim_{i\to\infty}\vert\vert (s_{i},n_{i})\vert\vert_{\infty}=\lim_{i\to\infty}\vert\vert (s_{i},1)\vert\vert_{\infty}=\limsup_{i\to\infty}s_{i}=\infty$.
This, however, is easy to see. As $s_{i}=X_{i}$ the definition of $s_{i}$ guarantees that the number
 $s_{i}=s_{i}(X_{i})$ minimizes $\max_{1\leq j\leq k} \vert\vert \zeta_{j}s_{i}\vert\vert$ among 
all $s_{i}\leq X_{i}$. On the other hand clearly $\liminf_{s\to\infty} \max_{1\leq j\leq k} \vert\vert \zeta_{j}\vert\vert=0$
for any $\mathbb{Q}$-linearly independent $\zeta_{1},\zeta_{2},\ldots,\zeta_{k}$
and so by definition of $(s_{i})_{i\geq 1}$ we also have $\lim_{i\to\infty}\max_{1\leq j\leq k} \vert\vert \zeta_{j}s_{i}\vert\vert=0$. 
Consequently the sequence $(s_{i})_{i\geq 1}$ 
cannot be bounded as only finitely many (strictly positive) values $\max_{1\leq j\leq k}\vert\vert \zeta_{j}s_{i}\vert\vert$ would
 appear, which proves $\limsup_{i\to\infty}s_{i}=\infty$.\\ 

\noindent To complete the proof we have to find a value $s_{0}=s_{0}(X)$ for which (\ref{eq:anelim}) holds.
Note first, that for sufficiently large $X$ und $s=s_{0}(X)$ we have $a_{1}^{j,(s)}=a_{1}^{\prime j,(s)}=1$.
Indeed for $s\geq \frac{1}{\min_{i}\vert\vert\zeta_{i}\vert\vert}$ and $i_{0}$ the index, for which the minimum is attained, 
 we have $\{s\zeta_{i_{0}}\} \notin{\{[0,\frac{1}{s}]\cup[\frac{s-1}{s},1)\}}$, so the first digit after the coma in base $s$
 is neither $0$ nor $(s-1)$.
So we can assume $X$ to be large enough to ensure $a_{1}^{j,(s)}=1$ for all $1\leq j\leq k$ and hence $b_{1}^{(s)}=1$ as well.
 It is now easy to see that putting $s_{0}:=s$ is an appropriate choice, since (\ref{eq:merz}) says that all $s\zeta_{j}$ 
respectively $s(1-\zeta_{j})$ start with $\lfloor\nu\rfloor$ digits zero in base $s$ behind the coma. This yields 
$\frac{a_{2}^{j,(s)}-a_{1}^{j,(s)}}{a_{1}^{j,(s)}}=a_{2}^{j,(s)}-1\geq \lfloor\nu\rfloor+1\geq \nu$ for all $1\leq j\leq k$,
therefore $\frac{b_{2}^{(s)}-b_{1}^{(s)}}{b_{1}^{(s)}}= \frac{\min_{j}a_{2}^{j,(s)}-a_{1}^{j,(s)}}{a_{1}^{j,(s)}} \geq \nu$
respectively the same facts for $a_{.}^{\prime .,(s)},b_{.}^{\prime (s)}$. 
\end{proof}
\noindent We easily deduce the following Corollary:

\begin{coro}  \label{korolar1}
We have
 \begin{equation*} 
\omega\geq  \max \left\{\sup_{s} \limsup_{n\geq 1} \frac{b_{n+1}^{(s)}-b_{n}^{(s)}-1} {b_{n}^{(s)}}, \sup_{s} \limsup_{n\geq 1} \frac{b_{n+1}^{\prime (s)}-b_{n}^{\prime (s)}-1} {b_{n}^{\prime (s)}}\right\}. 
\end{equation*}
\end{coro}
\noindent Similarly, we can give a lower bound for $\widehat{\omega}$ with respect to the $s$-adic representation of a real number.

\begin{prop}   \label{prop2}
For any $\zeta\in{\mathbb{R}}$ we have
\begin{equation*} 
\widehat{\omega}\geq \max \left\{\sup_{s} \liminf_{n\geq 1} \max_{1\leq j\leq n} \frac{b_{j+1}^{(s)}-b_{j}^{(s)}-1} {b_{n+1}^{(s)}}, \sup_{s} \liminf_{n\geq 1}  \max_{1\leq j\leq n} \frac{b_{j+1}^{\prime(s)}-b_{j}^{\prime(s)}-1} {b_{n+1}^{\prime(s)}}\right\}.
\end{equation*}
\end{prop}

\begin{proof}
By definition of the supremum it is suffient to prove 

\[
\widehat{\omega}\geq \mathscr{A}_{s}:=\max \left\{\liminf_{n\geq 1} \max_{1\leq j\leq n} \frac{b_{j+1}^{(s)}-b_{j}^{(s)}-1} {b_{n+1}^{(s)}}, \liminf_{n\geq 1}  \max_{1\leq j\leq n} \frac{b_{j+1}^{\prime(s)}-b_{j}^{\prime(s)}-1} {b_{n+1}^{\prime(s)}}\right\}
\]
for any base $s$ separately. So let $s$ be fixed and put $b_{n}^{(s)}=b_{n}$.
By definition of $\widehat{\omega}$ 
for arbitrary $\epsilon>0$ and sufficiently large $X=X(\epsilon)$ we have to find an approximation vector
 $(x,y_{1},\ldots ,y_{k})\in{\mathbb{Z}^{k+1}}$ with $x\leq X$ and 

\begin{equation} \label{eq:mileniza}
 \max_{1\leq j\leq k} \left\vert \zeta_{j}x-y_{j}\right\vert \leq X^{-\mathscr{A}_{s}+\epsilon}.
\end{equation}
For $\epsilon>0$ and large $X$ let $n_{0}$ be defined by $s^{b_{n_{0}}}\leq X< s^{b_{n_{0}+1}}$
respectively $s^{b_{n_{0}}^{\prime}}\leq X< s^{b_{n_{0}+1}^{\prime}}$. Put 
$x:=s^{b_{j}}$ respectively $x:=s^{b_{j}^{\prime}}$ where $j$ is the index, such that the inner maximum 
from the definition of $\mathscr{A}_{s}$ is attained for the given $n_{0}$.   
By definition of $(b_{n})_{n\geq 1}$ as the mixed sequence the first $b_{j+1}-b_{j}-1$ positions behind the comma
of each $\zeta_{t}s^{b_{j}}, 1\leq t\leq k$, respectively $(1-\zeta_{t})s^{b_{j}}, 1\leq t\leq k$,
 are zeros in base $s$. We infer that putting $y_{t}:=\lfloor \zeta_{t}x\rfloor$
for all $1\leq t\leq k$ respectively $y_{t}:= \lceil \zeta_{t}x\rceil$ for all $1\leq t\leq k$ we have

\begin{eqnarray}
  \max_{1\leq t\leq k}\left\vert\left\vert \zeta_{t}x\right\vert\right\vert= \max_{1\leq t\leq k} \left\vert \zeta_{t}x-y_{t} \right\vert \leq s^{b_{j}-b_{j+1}+1}&\leq& X^{\frac{b_{j}-b_{j+1}+1}{b_{n_{0}+1}}}  \label{eq:rumenige}  \\
\text{resp.} \quad \max_{1\leq t\leq k}\left\vert\left\vert \zeta_{t}x\right\vert\right\vert= \max_{1\leq t\leq k} \left\vert (1-\zeta_{t})x-y_{t} \right\vert \leq s^{b_{j}^{\prime}-b_{j+1}^{\prime}+1}&\leq& X^{\frac{b_{j}^{\prime}-b_{j+1}^{\prime}+1}{b_{n_{0}^{\prime}+1}}}.  \label{eq:rummenige}
\end{eqnarray}
For the left hand side inequalties compare the proof of Proposition \ref{prop1}, the right hand side inequalities follow 
from $X<s^{b_{n_{0}+1}}$ and $X<s^{b_{n_{0}+1}^{\prime}}$ respectively.
As (\ref{eq:rumenige}),(\ref{eq:rummenige}) holds for every large $X$, we may let $n$ tend to $\infty$
to conclude that (\ref{eq:mileniza}) has a solution for all sufficiently large $X$. Hence the exponent of $X$ in
(\ref{eq:rumenige}) and (\ref{eq:rummenige}) respectively is larger than $\mathscr{A}_{s}-\epsilon$. 
\end{proof}
\noindent Now we turn to simultaneous approximation of vectors $(\zeta_{1},\zeta_{2},\ldots ,\zeta_{k})$ with good approximation 
in one fixed simultaneous base $s\geq 2$, and we skip the dual expansion. We want to use Corollary \ref{korolar1} and 
Proposition \ref{prop2} to give estimates for the simultaneous approximation constants $\omega,\widehat{\omega}$. 
With respect to the notation above, meaning that the $s$-adic digits of $\zeta_{i}$ are given by $(a_{n,i}^{(s)})_{n\geq 1}$ as
 in (\ref{eq:ttt}) and the ordered mixed sequence by $(b_{n}^{(s)})_{n\geq 1}$, we get

\begin{lemm}  \label{lemma1}
 For any $s\geq 2$ we have
 \[
\min_{i} \liminf_{n\geq 1} \left(\frac{a_{n+1}^{i,(s)}}{a_{n}^{i,(s)}}\right)^{1/k}\leq \limsup_{n\geq 1} \frac{b_{n+1}^{(s)}} {b_{n}^{(s)}} \leq \min_{i} \limsup_{n\geq 1}\frac{a_{n+1}^{i,(s)}}{a_{n}^{i,(s)}}.
\]
\end{lemm}
 
\begin{proof}
The right hand inequality is trivial. For the left hand inequality keep $s$ fixed and put
 $C:=\min_{i} \liminf_{n\geq 1} \frac{a_{n+1}^{i,(s)}}{a_{n}^{i,(s)}}$ and choose $n_{0}$ large enough, such that for all $i$ and 
all $n\geq n_{0}$ we have  
\[
\frac{a_{n+1}^{i}}{a_{n}^{i}}\geq C-\epsilon
\]
($s$ has been dropped in the notation).
For arbitrary $b_{n}, n\geq n_{0}$, there exist $m,i_{0}$ with $a_{m}^{i_{0}}= b_{n}$ by definition of $(b_{n})_{n\geq 1}$. The 
interval $[a_{m}, (C-\epsilon)a_{m}]$ contains at most $k$ numbers $b_{i}$, since it contains at most one element of
every sequence $(a_{n}^{i})_{n\geq 1}$ for $1\leq i\leq k$.
By the pigeon hole principle there are two numbers $b_{j},b_{j+1}$ in the interval $[a_{m}^{i_{0}},a_{m+1}^{i_{0}}]$ 
whose quotient $\frac{b_{j+1}}{b_{j}}$ 
is at least $(C- \epsilon)^{1/k}$. The lemma follows with $\epsilon \to 0$.
\end{proof}
\noindent In combination with Corollary \ref{korolar1} we observe

\[
 \omega \geq \min_{i} \liminf_{n\geq 1} \left(\frac{a_{n+1}^{i,(s)}}{a_{n}^{i,(s)}}\right)^{1/k} -1, \qquad \forall s\geq 2.
\]  
Getting lower bounds for $\widehat{\omega}$ by just considering the sequences $a_{n+1}^{i,(s)}$ 
 is more complicated and to some extent impossilbe as we will see in Corollary \ref{korol1}. In fact even if 

\[
 \lim_{n\to\infty} \frac{a_{n+1}^{i,(s)}}{a_{n}^{i,(s)}} =\infty, \qquad 1\leq i\leq k,
\]
we can have $\widehat{\omega}=1/k$, which is the weakest lower bound for $\widehat{\omega}$ by (\ref{eq:2001neu}). We only mention that if we 
construct sequences $a_{n+1}^{i,(s)}$ for which $\lim_{n\to\infty} \frac{b_{n+1}^{(s)}}{b_{n}^{(s)}}=\infty$, Proposition \ref{prop2} yields

\[
 \widehat{\omega}\geq \liminf_{n\geq 1}\frac{b_{n+1}^{(s)}-b_{n}^{(s)}}{b_{n+1}^{(s)}} = 1- \frac {1}{\liminf_{n\geq 1}\frac{b_{n+1}^{(s)}}{b_{n}^{(s)}}}= 1,
\]
and consequently $\widehat{\omega}=1$ in view of (\ref{eq:2001neu}).  

\subsection{The case $\omega=\infty$}
In the following theorem, we 
compute the classical approximation constants $\omega_{j},\widehat{\omega}_{j}$ for a special type of Liouville numbers 
$\zeta_{1},\zeta_{2},\ldots ,\zeta_{k}$, whose
 best approximation vectors $(x,y_{1},y_{2},\ldots ,y_{k})$ to $(\zeta_{1},\zeta_{2},\ldots ,\zeta_{k})$
are easy to guess. The main arguments of the compilation will be carried out in the proofs of the following theorems.

\begin{theo}  \label{satz1}
 Let $k$ be a positive integer and for $1\leq j\leq k$ let $\zeta_{j}= \sum_{n\geq 1} \frac{1}{q_{n,j}}$, where

\begin{equation} \label{eq:qq}
 q_{1,1}<q_{1,2}< \ldots <q_{1,k}<q_{2,1}<q_{2,2}<\ldots q_{2,k}<q_{3,1}<\ldots
\end{equation}
are natural numbers, such that 

\begin{equation} \label{eq:erwin}
 q_{n,j}\vert q_{n,j+1} \quad \text{for}\quad 1\leq j\leq k-1 \quad \text{and} \quad q_{n,k}\vert q_{n+1,1} \quad \text{for all} \quad n\geq 1
\end{equation}
and such that

\begin{eqnarray}
 \lim_{n\to\infty} \frac{\log(q_{n+1,1})-\log(q_{n,k})}{\log(q_{n+1,k})}&=& \eta_{1},   \label{eq:nurrrrr}   \\ 
 \lim_{n\to\infty} \frac{\log(q_{n+1,i})-\log(q_{n+1,i-1})}{\log(q_{n+1,k})}&=& \eta_{i}, \quad 2\leq i\leq k,   \label{eq:nurrrr} \\
 \lim_{n\to\infty} \frac{\log(q_{n+1,1})}{\log(q_{n,k})}&=& \eta_{k+1}=\infty,  \label{eq:nurrr} 
\end{eqnarray}
where $\eta=(\eta_{1},\eta_{2},\ldots ,\eta_{k+1})\in{\mathbb{R}^{k}\times\overline{\mathbb{R}}}$ satisfy

\begin{eqnarray}
&&  \eta_{1}+\eta_{2}+\cdots + \eta_{k}=1        \label{eq:19neu} \\
&& \eta_{k+1}>\eta_{k}\geq \eta_{k-1} \geq \ldots \geq \eta_{1}> 0   \label{eq:18neu} \\
&& \eta_{k+1}= \infty.     \label{eq:17neu}  
\end{eqnarray}
Then the classical approximation constants relative to the vector $\boldsymbol{\zeta}=(\zeta_{1},\zeta_{2},\ldots ,\zeta_{k})$ are given by

\begin{eqnarray*}
\omega_{1} & = & \eta_{k+1}= \infty=:\wp_{1}(\eta)     \\
\omega_{2} & = & \max \left\{\frac{\eta_{k}}{\eta_{k}+\eta_{k-1}+\cdots +\eta_{1}}, \frac{\eta_{k-1}}{\eta_{k-1}+\eta_{k-2}+\cdots +\eta_{1}},\ldots ,\frac{\eta_{1}}{\eta_{1}}\right\}=:\wp_{2}(\eta)   \\
\omega_{3} & = & \max \left\{\frac{\eta_{k-1}}{\eta_{k}+\eta_{k-1}+\cdots +\eta_{1}}, \frac{\eta_{k-2}}{\eta_{k-1}+\eta_{k-2}+\cdots +\eta_{1}},\ldots ,\frac{\eta_{1}}{\eta_{2}+\eta_{1}}\right\}=:\wp_{3}(\eta)   \\
\vdots && \vdots                \\
\omega_{k+1} & = & \frac{\eta_{1}}{\eta_{k}+\eta_{k-1}+\cdots +\eta_{1}}=:\wp_{k+1}(\eta).
\end{eqnarray*}
and 

\begin{eqnarray*}
\widehat{\omega}_{1}&=& \min \left\{\frac{\eta_{k}}{\eta_{k}+\eta_{k-1}+\cdots +\eta_{1}}, \frac{\eta_{k-1}}{\eta_{k-1}+\eta_{k-2}+\cdots +\eta_{1}},\ldots ,\frac{\eta_{1}}{\eta_{1}}\right\}=:\widehat{\wp}_{1}(\eta)   \\
\widehat{\omega}_{j}&=& 0, \qquad 2\leq j\leq k+1.
\end{eqnarray*}
\end{theo}

\begin{proof}
We start with the constants $\omega_{j}$ and intend to prove the inequalities
 $\omega_{j}\geq \wp_{j}(\eta)$ and $\omega_{j}\leq \wp_{j}(\eta)$ seperately for $1\leq j\leq k+1$. \\

$\underline{\omega_{j}\geq \wp_{j}(\eta)}:$   \\
Let $\wp_{j,l}$ be the $l$-th quotient of the maximum labeled $\wp_{j}(\eta)$. We give a detailed proof
 of $\omega_{j}\geq \wp_{j,1}=\eta_{k+2-j}$ and then
mention how to generalize the proof to derive all the other inequalities $\omega_{j}\geq \wp_{j,l}$ for $l\neq 1$.\\
To prove $\omega_{j}\geq \wp_{j,1}$, we will construct $j$ sequences of approximation vectors

\[
 \left(x^{(1,i)},y_{1}^{(1,i)},\ldots  ,y_{k}^{(1,i)}\right)_{i\geq 1},
\left(x^{(2,i)},y_{1}^{(2,i)},\ldots y_{k}^{(2,i)}\right)_{i\geq 1},\ldots \left(x^{(j,i)},y_{1}^{(j,i)},\ldots ,y_{k}^{(j,i)}\right)_{i\geq 1}
\]
which are linearly independent for each fixed $i\in{\mathbb{N}}$ and such that
$\omega_{j}=\eta_{k+2-j}$ follows for $i\to\infty$. Indeed for $p$ in a $j$-element subset of $\{1,2,\ldots ,k\}$ and 
any $\epsilon>0$ we claim for $i$ sufficiently large 

\[
 \max_{1\leq t\leq k} -\frac{\log \left(\left\vert \zeta_{t}x^{(p,i)}-y_{t}^{(p,i)}\right\vert\right)}{\log\left(x^{(p,i)}\right)}\geq \eta_{k+2-j}-\epsilon.
\]
In analogy to the definition of $(b_{n}^{s})_{n\geq 1}$ in subsection 2.1 let $(b_{n})_{n\geq 1}$ be the combined sequence of the logarithms of the integers $q_{n,j}$ in increasing
order, which means for any nonnegative integer $M$ and $N\in{\{1,2,\ldots ,k\}}$ we have $b_{kM+N}=\log(q_{M,N})$. By (\ref{eq:nurrr})
we have $\limsup \frac{b_{n+1}}{b_{n}}=\infty$ and thus by putting the first approximation vector
$(q_{n,i},\lfloor \zeta_{1}q_{n,i}\rfloor,\ldots, \lfloor \zeta_{k}q_{n,i}\rfloor)$ with arbitrary $i$ we may let $n$ tend to infinity,
to obtain $\omega=\infty$: indeed applying (\ref{eq:erwin}) we derive that all the remainder terms

\[
 \vert\vert \zeta_{j}q_{n,i}\vert\vert= \sum_{l:q_{n,l}>q_{n,i}} \frac{1}{q_{l,j}}q_{n,i}\leq 2\frac{q_{n,i}}{q_{n,i+1}}
\]
are small due to (\ref{eq:nurrrrr})-(\ref{eq:nurrr}).
In order to estimate $\omega_{j}$ for $j\geq 2$ we construct a sequence of parameters $X$ and approximation vectors
$(x,y_{1},\ldots,y_{k})$ with $x\leq X$ explicitely.
For the fixed choice $X^{(n)}:= q_{n,k}$ we will get $\omega_{j}\geq \wp_{j,1}$. To see this let $x^{(1,n)}:=X{(n)}$ and 
$y_{t}^{(1,n)}:=\lfloor x^{(1,n)}\zeta_{t}\rfloor$ for $1\leq t\leq k$. Define the second approximation vector by 
taking $x^{(2,n)}=q_{n,k-1}$ and again
$y_{t}^{(2,n)}:=\lfloor x^{(2,n)}\zeta_{t}\rfloor$. By means of (\ref{eq:nurrrrr}),(\ref{eq:nurrrr}),(\ref{eq:nurrr})
 and the definition of $\wp_{2}(\eta)$ we claim that for each $C<\eta_{k+2-2}=\eta_{k}$ 

\begin{equation} \label{eq:mota}
 \left\vert \zeta_{t}x^{(2,n)}- y_{t}^{(2,n)}\right\vert \leq \left(x^{(1,n)}\right)^{-C}= \left(X^{(n)}\right)^{-C}
\end{equation}
holds for $n=n(C)$ large enough. This follows from 

\begin{eqnarray*}
 \left\vert \zeta_{t}x^{(2,n)}- y_{t}^{(2,n)}\right\vert&=& \left\vert q_{n,k-1}\zeta_{t}- \lfloor q_{n,k-1}\zeta_{t}\rfloor \right\vert=\sum_{i=n+1}^{\infty} \frac{q_{n,k-1}}{q_{i,t}}, \quad 1\leq t\leq k-1, \\
 \left\vert \zeta_{k}x^{(2,n)}- y_{k}^{(2,n)}\right\vert&=& \left\vert q_{n,k-1}\zeta_{k}- \lfloor q_{n,k-1}\zeta_{k}\rfloor \right\vert=\frac{q_{n,k-1}}{q_{n,k}}+\sum_{i=n+1}^{\infty}\frac{q_{n,k-1}}{q_{i,k}} 
\end{eqnarray*}
in view of the definition of $\zeta_{t}$ and our assumption (\ref{eq:erwin}). In every case
all the values of $\zeta_{t}x^{(2,n)}-y_{t}^{(2,n)}$ for $1\leq t\leq k$ are bounded by
$\frac{q_{n,k-1}}{q_{n,k}}(1+o(1))$. Using (\ref{eq:nurrrrr}),(\ref{eq:nurrrr}),(\ref{eq:nurrr}) this leads to (\ref{eq:mota}). \\  

\noindent Similarly, defining the $j$-th approximation vector for $2\leq j\leq k$
 by $x^{(j,n)}=q_{n,k+1-j}$ and for $j=k+1$ by $x^{(k+1,n)}=q_{n-1,k}$ and then putting  
$y_{t}^{(j,n)}:=\lfloor x^{(j,n)}\zeta_{t}\rfloor$ yields the corresponding inequalities. \\
We now check that these vectors are linearly independent as required. To do this we prove that all the matrices 
$B_{n}=(B_{n}(i,j))_{1\leq i,j\leq k+1}$ obtained by writing the $h$-th approximation vector
$(x^{(h,n)},y_{1}^{(h,n)},\ldots ,y_{k}^{(h,n)})$ in the $h$-th row, i.e. 

\begin{displaymath}
\boldsymbol{B_{n}}= \left( \begin{array}{cccccccccc}
q_{n,k} & q_{n,k}\sum_{i=1}^{n}q_{i,1}^{-1} & q_{n,k}\sum_{i=1}^{n}q_{n,2}^{-1} & \ldots & q_{n,k}\sum_{i=1}^{n}q_{i,k}^{-1}      \\
q_{n,k-1} & q_{n,k-1}\sum_{i=1}^{n}q_{i,1}^{-1} & q_{n,k-1}\sum_{i=1}^{n}q_{n,2}^{-1} & \ldots & q_{n,k-1}\sum_{i=1}^{n-1}q_{i,k}^{-1}  \\
\vdots & \vdots  & \vdots & \vdots & \vdots\\
q_{n,1} & q_{n,1}\sum_{i=1}^{n}q_{i,1}^{-1}& q_{n,1}\sum_{i=1}^{n-1}q_{n,k-1}^{-1} & \ldots & q_{n,1}\sum_{i=1}^{n-1}q_{i,k}^{-1}    \\
q_{n-1,k} & q_{n-1,k}\sum_{j=1}^{n-1}q_{j-1,1}^{-1} & q_{n-1,k}\sum_{j=1}^{n-1}q_{j-1,2}^{-1} & \ldots & q_{n-1,k}\sum_{j=1}^{n-1}q_{j-1,k}^{-1}
\end{array} \right)
\end{displaymath}

are nonsingular. Observe that if we subtract $\frac{B_{n}(h,1)}{B_{n}(h+1,1)}$ times the $(h+1)$-th row
from the $h$-th row of the matrix $B_{n}$, 
all entries in the new $h$-th line will be zero apart
from a one in position $(h,k+2-h)$. Starting with this process at $h=1$ and repreating it until $h=k$ we end up with the matrix

\begin{displaymath}
\boldsymbol{C_{n}}= \left( \begin{array}{cccccccccc}
0 &  0 & \ldots & 0 & 1      \\
0 & \ldots & 0 & 1 & 0   \\
\vdots & \vdots  & \reflectbox{$\ddots$} & \vdots & \vdots \\
0 &  1 & 0 & \ldots & 0   \\
q_{n-1,k} & q_{n-1,k}\sum_{j=1}^{n-1}q_{j-1,1}^{-1}& q_{n-1,k}\sum_{j=1}^{n-1}q_{j-1,2}^{-1} & \ldots & q_{n-1,k}\sum_{j=1}^{n-1}q_{j-1,k}^{-1}
\end{array} \right),
\end{displaymath}
which is easily seen to have absolute value of the determinant equal to $q_{n-1,k}\neq 0$. Therefore also $\det(B_{n})=q_{n-1,k}\neq 0$, as required. \\

\noindent To obtain all the other inequalities $\omega_{j}\geq \wp_{j,i}$ for $2\leq i$, where the upper bound of $i$ depends on $j$,
we proceed analoguously.
In the definition of $x^{(1,n)}$ we replace $q_{n,k}$ by $q_{n,k+1-i}$ and again for $1\leq t\leq k$ we define 
$y_{t}^{(1,n)}=\lfloor x^{(1,n)}\zeta_{t}\rfloor$ for the first approximation vector. We define all the others
by taking $x^{(2,t)}=q_{n,k+1-2},x^{(3,n)}=q_{k+1-3,n},\ldots$ and again $y_{t}^{(i,n)}=\lfloor x^{(i,n)}\zeta_{t}\rfloor$ for
$1\leq t\leq k$ and $2\leq i\leq k+1$. This construction yields the desired lower bounds (or $0$ which is omitted in $\wp(\eta)$) 
as above again by (\ref{eq:nurrrrr}),(\ref{eq:nurrrr}),(\ref{eq:nurrr}).\\

$\underline{\omega_{j}\leq \wp_{j}(\eta)}:$   \\
We have to show that for $1\leq j\leq k+1$ the approximation vectors $\left(x^{(j,n)},y_{1}^{(j,n)},\ldots y_{k}^{(j,n)}\right)$ 
constructed in the first step of the proof are somehow best possible. We split the proof of this assertion in 3 steps. 
To simplify notation let $(c_{n})_{n\geq 1}=(e^{b_{n}})_{n\geq 1}$ be the ordered mixed sequence
 $(q_{1,1},q_{1,2},\ldots ,q_{1,k},q_{2,1},\ldots)$.\\
First step: For an arbitrary approximation vector $(x,y_{1},\ldots ,y_{k})$ let $h$ be the index determined by $c_{h}\leq x<c_{h+1}$ and
let $g$ be the largest integer such that the index $g-1$ satisfies $c_{g-1}\vert x$. Since $x<c_{h+1}$ 
and consequently $c_{h+1}\nmid x$ we clearly have $g\leq h+1$. 
When $X\to\infty$ so does $h$ and we claim that for $h\to\infty$

\begin{eqnarray}
\max_{1\leq t\leq k} \left\vert \zeta_{t}x-y_{t}\right\vert &\geq& \frac{c_{g-1}}{c_{g}}-o\left(\frac{c_{g-1}}{c_{g}}\right), \qquad g> h+1-k   \label{eq:mimena} \\
\max_{1\leq t\leq k} \left\vert \zeta_{t}x-y_{t}\right\vert &\geq& \frac{1}{c_{h+2-k}}-o\left(\frac{1}{c_{h+2-k}}\right),  \qquad g\leq h+1-k.  \label{eq:memina}
\end{eqnarray}
Furthermore in the case $g\leq h+1-k$ (i.e. the assumption of (\ref{eq:memina})), the inequality
$x< \frac{1}{2}c_{h+1}c_{h+1-k}^{-1}$ contradicts that 

\begin{equation}  \label{eq:wilena}
\max_{1\leq t\leq k} \left\vert \zeta_{t}x-y_{t}\right\vert < \frac{1}{2}\frac{1}{c_{h+1-k}}-o\left(\frac{1}{c_{h+1-k}}\right)
\end{equation}
holds for $h\to\infty$.\\

\noindent Second step: Let $X$ be a real parameter from the definition of the approximation constants $\omega_{j}$ and $m=m(X)$ be the index
such that $c_{m}\leq X<\frac{c_{m+1}}{4}$. Then for $1\leq j\leq k+1$ a set of $j$ vectors $(x^{(i)},y_{1}^{(i)},\ldots ,y_{k}^{(i)}), 1\leq i\leq j$,
satisfying the inequalities

\begin{eqnarray}
 x^{(i)} &\leq&  X \qquad   1\leq i\leq j        \label{eq:28neu}\\
\max_{1\leq t\leq k} \left\vert x^{(i)}\zeta_{t} - y_{t}^{(i)}\right\vert &\leq& \frac{1}{2}, \qquad 1\leq i\leq j, \label{eq:29neu} 
\end{eqnarray}
 can only be linearly independent if at least one $x^{(i)}$ is not divisable by $c_{m+2-j}$. \\

\noindent Third step: We intend to show by combining the first two steps and using (\ref{eq:nurrrrr})-(\ref{eq:nurrr}),
 that for arbitrary $X$ the choice of approximation vectors in the proof of
$\omega_{j}\geq \wp_{j}(\eta)$ is somehow optimal, i.e. the approximation constants of this case cannot be improved.\\

\noindent Proof of first step: We first make the assumption $c_{h+1}\in{(q_{n,k})_{n\geq 1}}$, and will explain at the end how to extend
this easily to the case where $c_{h+1}$ belongs to another sequence. This assumption is equivalent to $c_{h+1}=q_{n_{1},k}$ for some
$n_{1}\in{\mathbb{N}}$ and it follows that $c_{h}=q_{n_{1},k-1}$. 
By (\ref{eq:erwin}) we have $c_{l}\vert x$ for all $l\leq g-1$ and 
$c_{l}\nmid x$ for all $l\geq g$, in particular $c_{g}\nmid x$. 
Recall $g\leq h+1$. To prove the assertions we now consider the corresponding cases seperately: \\

\noindent Case 1: $c_{g}> q_{n_{1}-1,k}$. Note that since $q_{n_{1}-1,k}=c_{h+1-k}$ this is equivalent to $g\geq h+2-k$ or $c_{g}\geq c_{h+2-k}$.
 We can write $x=x_{1}+x_{2}$ with $0<x_{1}<c_{g}$ and $c_{g}\vert x_{2}$, 
since by our definition of $g$ we have $x_{1}\neq 0$. Denote by $\overline{g}$ the congruence class
of $g$ in the residue system $\{1,2,\ldots ,k\}$ mod $k$. Note, that $c_{\overline{g}}$ is the smallest value $c_{g}$ 
with $g$ in the residue class $\overline{g}$, or equivalently $\zeta_{\overline{g}}=\sum_{l\geq 0} c_{\overline{g}+kl}^{-1}$, which
we will make use of. We claim that

\begin{eqnarray}
  \left\vert\left\vert x_{1}\zeta_{\overline{g}}\right\vert\right\vert &\geq& c_{g}^{-1}c_{g-1}-\sum_{l\geq 1}c_{g+lk}^{-1}c_{g-1}\geq c_{g}^{-1}c_{g-1}- 2c_{h+1}c_{h+2}^{-1}      \label{eq:trunschuh}   \\
 \{x_{2}\zeta_{\overline{g}}\}&=& \left\vert\left\vert x_{2}\zeta_{\overline{g}}\right\vert\right\vert= \left\vert\left\vert x_{2}\sum_{l\geq 1} c_{g+kl}^{-1}\right\vert\right\vert\leq 2c_{h+1}c_{h+2}^{-1},    \label{eq:truhnschuh}
\end{eqnarray}
where $\vert\vert\ . \vert\vert$ denotes the closest distance to an integer an $\{.\}$ the fractional part of a real number.
Inequality (\ref{eq:trunschuh}) relies on the fact that $0<\frac{x_{1}}{c_{g}}<1$ and $c_{g-1}\vert x_{1}$, which is seen to be true
 because
$c_{g-1}\vert x,c_{g-1}\vert c_{g}$ and $c_{g}\vert x_{2}$ by definition, so putting these together we get $c_{g-1}\vert x-x_{2}$, 
but $x-x_{2}=x_{1}$. Combination of these two facts and recalling that exactly $g+k,g+2k,\ldots $ are the indices greater than 
$g$ belonging to the residue class $\overline{g}$ shows that $x_{1}\zeta_{\overline{g}}$ is of the form

 \[
  x_{1}\zeta_{\overline{g}}=\frac{x_{1}}{c_{g}}+x_{1}\sum_{l\geq 1}c_{g+lk}^{-1}=\frac{Kc_{g-1}}{c_{g}}+x_{1}\sum_{l\geq 1}c_{g+lk}^{-1}\geq K\frac{c_{g-1}}{c_{g}}+x_{1}\sum_{l\geq 1}c_{g+lk}^{-1}
 \]
with $K\in{\{1,2,\ldots ,\frac{c_{g}}{c_{g-1}}-1\}}$ (note $c_{g-1}\vert c_{g}$). 
The assertion now follows by a combination of $x_{1}<c_{g}\leq c_{h+1}$,  

\begin{equation} \label{eq:pfred}
 \sum_{l\geq 1} c_{g+lk}^{-1}< c_{g+k}^{-1}\left(1+\frac{1}{2}+\frac{1}{4}+\cdots \right)= 2c_{g+k}^{-1},
\end{equation}
and $c_{g+k}^{-1}\leq c_{h+2}^{-1}$, which is true by the assumption of case 1.\\
Inequality (\ref{eq:truhnschuh}) follows from the fact that for any $s\leq g$ by virtue of (\ref{eq:erwin}) we have
$c_{s}\vert x_{2}$ which holds in particular for those $c_{s}$ with $s$ in the residue class $\overline{g}$. 
So all quantities $x_{2}c_{s}^{-1}$ with $s\leq g$ are integers.
Thus the sum of quantities of order smaller than $x_{2}c_{g+k}^{-1}$
in $x_{2}\zeta_{\overline{g}}=\sum_{l\geq 0} x_{2}c_{\overline{g}+kl}^{-1}$,
 i.e. $x_{2}\sum_{l\geq 1} c_{g+kl}^{-1}$, has the same fractional part as the entire sum. Now on the one hand we have
$x_{2}\leq x\leq c_{h+1}$, and on the other hand $c_{g+k}\geq c_{h+2}$ by the assumption of case 1. Together with (\ref{eq:pfred})
these assertions yield (\ref{eq:truhnschuh}).\\

\noindent Summing (\ref{eq:trunschuh}) and (\ref{eq:truhnschuh}) and noting that by (\ref{eq:nurrr}) we have
$c_{h+1}c_{h+2}^{-1}=o(c_{g}^{-1}c_{g-1})$, so that we can use the triangle inequality on the fractional parts,
 we further have

\[
\vert\vert x\zeta_{\overline{g}}\vert\vert \geq c_{g}^{-1}c_{g-1}-4c_{h+1}c_{h+2}^{-1}.
\]
By (\ref{eq:18neu}),(\ref{eq:17neu}) and (\ref{eq:nurrr}) the expression $c_{m}^{-1}c_{m-1}$ is 
montonically decreasing in $m$ and the error term $4c_{h+1}c_{h+2}^{-1}$ is obviously
$o(c_{h}c_{h+1}^{-1})$ by (\ref{eq:nurrr}). Hence for $h\to\infty$ we obtain

\[
 \frac{c_{g-1}}{c_{g}}-o\left(\frac{c_{h}}{c_{h+1}}\right)\leq \vert\vert \zeta_{\overline{g}}x\vert\vert\leq \max_{1\leq t\leq k}\vert\vert \zeta_{t}x\vert\vert. 
\]
This establishes (\ref{eq:mimena}) in this case as $\frac{c_{h}}{c_{h+1}}\leq \frac{c_{g-1}}{c_{g}}$ by (\ref{eq:18neu})
and (\ref{eq:nurrrrr})-(\ref{eq:nurrr}). \\
\noindent If $c_{h+1}$ belongs to another sequence $(q_{n,i})_{n\geq 1}$, which means $c_{h+1}=q_{n_{1},i}$ with $i\neq 1$,
 we look at the case $c_{g}\geq q_{n_{1}-1,i}$ and again obtain

\begin{eqnarray*}
  \left\vert\left\vert x_{1}\zeta_{\overline{g}}\right\vert\right\vert &\geq& c_{g}^{-1}c_{g-1}-\sum_{l\geq 1}c_{g+lk}^{-1}c_{g-1}\geq c_{g}^{-1}c_{g-1}- 2c_{h+1}c_{g+k}^{-1}       \\
 \{x_{2}\zeta_{\overline{g}}\}&=& \left\vert\left\vert x_{2}\zeta_{\overline{g}}\right\vert\right\vert= \left\vert\left\vert x_{2}\sum_{l\geq 1} c_{g+kl}^{-1}\right\vert\right\vert\leq 2c_{h+1}c_{g+k}^{-1},  
\end{eqnarray*}
as in the proof of the special case (without using $c_{h+2}\leq c_{g+k}$ as above from which we derived the weaker
but sufficient conditions (\ref{eq:trunschuh}),(\ref{eq:truhnschuh})).
However, by (\ref{eq:nurrrrr})-(\ref{eq:nurrr}) we again have
 $c_{h+1}c_{g+k}^{-1}=o(c_{g}^{-1})$ as $h\to \infty$ (or equivalently $g\to\infty$ as $g\geq h+2-k$) 
 and the rest of the argumentation is almost as above. Thus (\ref{eq:mimena}) holds in any 
case.\\
\noindent Case 2: $c_{g}\leq q_{n_{1}-1,k}$. In this case it is more convenient to work directly with the values $q_{.,.}$ 
instead of $c_{.}$. As in case 1 let $x=x_{1}+x_{2}$  with $0<x_{1}<q_{n_{1},1}$ and 
$q_{n_{1},1}\vert x_{2}$. Again $q_{n_{1}-1,k}\vert q_{n_{1},1}$ and the definition of $g$ ensures $x_{1}\neq 0$.
Analoguously to the proof of (\ref{eq:trunschuh}),(\ref{eq:truhnschuh}) in case 1 we deduce 

\begin{eqnarray*}
 \left\vert\left\vert x_{1}\zeta_{1}\right\vert\right\vert &\geq& \frac{1}{q_{n_{1},1}}-2q_{n_{1},k}q_{n_{1}+1,1}^{-1}           \\
 0\leq x_{2}\zeta_{1}&\leq& 2q_{n_{1},k}q_{n_{1}+1,1}^{-1}.
\end{eqnarray*}
Using again the triangle inequality and (\ref{eq:17neu}), we again deduce

\[
 \left\vert\left\vert x\zeta_{1}\right\vert\right\vert \geq \frac{1}{q_{n_{1},1}}-4\frac{q_{n_{1},k}}{q_{n_{1}+1,1}}= \frac{1}{q_{n_{1},1}}-4\frac{c_{h+1}}{c_{h+2}}.
\]
But by (\ref{eq:nurrrrr})-(\ref{eq:nurrr}) again $\frac{c_{h+1}}{c_{h+2}}=o(c_{h+2-k}^{-1})=o(c_{n_{1},1}^{-1})$ for $h\to\infty$
so that finally

\[
 \frac{1}{q_{n_{1},k}}-o(c_{h+2-k}^{-1})\leq \vert\vert \zeta_{1}x\vert\vert \leq \max_{1\leq t\leq k} \vert\vert \zeta_{t}x\vert\vert.
\]
But $q_{n_{1},1}=c_{h+2-k}$, so we have (\ref{eq:memina}) in this case.
If $c_{h+1}$ belongs to another sequence $(q_{n,i})_{n\geq 1}$, $i\neq k$, we can apply very similar estimates with respect to 
$\zeta_{\overline{i+1}}=\zeta_{i+1}$ instead of $\zeta_{1}$.
So our assumption is no loss of generality in this case either. Thus (\ref{eq:memina}) holds in any case. \\

\noindent We still have to prove that $x<\frac{1}{2}c_{h+1}c_{h+1-k}^{-1}$ contradicts 
(\ref{eq:wilena}). For simplicity we again discuss the case $c_{h+1}\in{(q_{n,k})_{n\geq 1}}$ first.
 Write $x=x_{1}+x_{2}$ with $0<x_{1}<q_{n_{1}-1,k}$ and $q_{n_{1}-1,k}\vert x_{2}$.
Note that again we have $x_{1}\neq 0$ by the assumption $g\leq h+1-k$, so $c_{g}\leq c_{h+1-k}=q_{n_{1}-1,k}$, and the definition of $g$. 
Assume we have $x<\frac{1}{2}c_{h+1}c_{h+1-k}^{-1}=\frac{1}{2}q_{n_{1},k}q_{n_{1}-1,k}^{-1}$.
The fractional part of $x_{2}\zeta_{k}$ is $\sum_{l\leq 0}x_{2}c_{h+1+lk}^{-1}$ as higher order
summands are integers by definition of $x_{2}$. We split this expression in
$\{x_{2}\zeta_{k}\}=\vert\vert x_{2}\zeta_{k}\vert\vert=\sum_{l\geq 0} x_{2}c_{h+1+kl}^{-1}=x_{2}c_{h+1}^{-1}+\sum_{l\geq 1}x_{2}c_{h+1+kl}^{-1}$
and using $x_{2}\leq x$ we infer

\begin{equation} \label{eq:elenard}
 \vert\vert x_{2}\zeta_{k}\vert\vert \leq \frac{1}{2}c_{h+1-k}^{-1}+\sum_{l\geq 1}c_{h+1}c_{h+1-k}^{-1}c_{h+1+kl}^{-1},  
\end{equation}
which is obviously $\frac{1}{2}c_{h+1-k}^{-1}+o(c_{h+1-k}^{-1})$ as $h\to\infty$ by (\ref{eq:nurrr}).\\
\noindent On the other hand, by definition of $x_{1}$ and
$c_{h+1-k}=q_{n_{1}-1,k}\nmid x_{1}$ as $g\leq h+1-k$ by assumption and a very similar argument as in case 1 we have

\[
 \left\vert\left\vert x_{1}\left(\frac{1}{q_{1,k}}+\frac{1}{q_{2,k}}+\cdots +\frac{1}{q_{n_{1}-1,k}}\right)\right\vert\right\vert\geq \frac{1}{q_{n_{1}-1,k}}.
\]
On the other hand by $x_{1}<q_{n_{1}-1,k}$ the sum of the remainder terms of $\zeta_{k}x_{1}$, i.e. 
$x_{1}\sum_{l\geq 0} \frac{1}{q_{n_{1}+l,k}}$, is
 bounded above by $2\frac{q_{n_{1}-1,k}}{q_{n_{1},k}}$ with very similar estimates as in (\ref{eq:pfred}). So

\begin{equation} \label{eq:dranele}
 \vert\vert x_{1}\zeta_{k}\vert\vert \geq \frac{1}{q_{n_{1}-1,k}}-2\frac{q_{n_{1}-1,k}}{q_{n_{1},k}}
\end{equation}
by a very similar argument as in case 1. As $n_{1}\to\infty$, we have 
$\frac{q_{n_{1}-1,k}}{q_{n_{1},k}}=o\left(\frac{1}{q_{n_{1}-1,k}}\right)=c_{h+1-k}^{-1}-o(c_{h+1-k}^{-1})$ 
(note that $h\to\infty$ if $n_{1}\to\infty$) and thus by (\ref{eq:dranele})

\begin{equation} \label{eq:drenale}
 \vert\vert x_{1}\zeta_{k}\vert\vert \geq c_{h+1-k}^{-1}-o(c_{h+1-k}^{-1}) 
\end{equation}
Using triangular inequality on (\ref{eq:elenard}),(\ref{eq:drenale}) thus gives

\[
 \max_{1\leq t\leq k}\vert\vert \zeta_{t}x\vert\vert \geq \vert\vert \zeta_{k}x\vert\vert \geq \frac{1}{2}c_{h+1-k}^{-1}-o(c_{h+1-k}^{-1}).
\]
Our last assertion is proved in this case and the assumption $c_{h}\in{(q_{n,k})_{n\geq 1}}$ can obviously be dropped again. \\

\noindent Proof of second step: Without loss of generality assume that $c_{m}\in{(q_{n,k})_{n\geq 1}}$, the proof for the other cases 
is essentially the same. This means $c_{m}=q_{m_{1},k}$ for some $m_{1}\in{\mathbb{N}}$ and consequently
$c_{m+1}=q_{m_{1}+1,1}$. Suppose $c_{m-j+2}$ divides $x^{(i)}$ for all $1\leq i\leq j$. 
On the one hand, by our assumption the $(j-1)$ numbers
 $c_{m-j+2},c_{m-j+3},\ldots ,c_{m}$ belong to the sequences $q_{n,k},q_{n,k-1},\ldots ,q_{n,k-j+2}$.
On the other hand, $c_{u}\vert c_{u+1}$ for all
$u\geq 1$ combined with $c_{m-j+2}\vert x^{(i)}$ for all $1\leq i\leq j$
implies that for all $s\leq m-j+2$ the number $c_{s}$ divides $x^{(i)}$. From these two facts we conclude
that for $g\notin{\{k,k-1,\ldots ,k+2-j\}}$, i.e.
$g\in{G:=\{1,2,\ldots ,k+1-j\}}$, the partial sum $x^{(i)}\sum_{r=1}^{m_{1}} \frac{1}{q_{r,g}}$ of $\zeta_{g}x^{(i)}$ is an integer,
since every summand $\frac{x^{(i)}}{q_{r,g}}$ is. As terms of order lower than $m_{1}$ in $\zeta_{g}x^{(i)}$ for
 $g\in{G\setminus{\{1\}}}$ obviously add up to a quantity smaller than $\frac{1}{2}$,
for $1\leq i\leq j$ and $g\in{G\setminus{\{1\}}}$ we have

\begin{equation}
\left\vert\left\vert \zeta_{g}x^{(i)}\right\vert\right\vert = \zeta_{g}x^{(i)}- \left\lfloor \zeta_{g}x^{(i)} \right\rfloor= \sum_{r=m_{1}+1}^{\infty} \frac{1}{q_{r,g}} < \frac{1}{2},  \label{eq:haarz} 
\end{equation}
and combined with (\ref{eq:29neu}) eventually

\begin{eqnarray}
 y_{g}^{(i)}&=&\left\lfloor \zeta_{g}x^{(i)} \right\rfloor = x^{(i)}\sum_{r=1}^{m_{1}} \frac{1}{q_{r,g}}.   \label{eq:zraah}            
\end{eqnarray}
In view of our assumption $X< \frac{c_{m+1}}{4}=\frac{q_{m_{1}+1,1}}{4}$ the results 
(\ref{eq:haarz}),(\ref{eq:zraah}) are also valid for $g=1$.
To sum up, for all $g\in{G}$ we have (\ref{eq:zraah}), which obviously yields

\[
 \frac{x^{(a)}}{x^{(b)}}= \frac{y_{g}^{(a)}}{y_{g}^{(b)}},\qquad g\in{G},\quad 1\leq a,b\leq j.
\]
Thus in the matrix, whose $i$-th row is the $i$-th approximation vector
$(x^{(i)},y_{1}^{(i)},\ldots ,y_{k}^{(i)})\in{\mathbb{Z}^{k+1}}$ ($1\leq i\leq j$), the first
$\vert G\vert=k-j+3$ columns together have rank $1$. The rank of the whole matrix therefore cannot exceed $1+[(k+1)-(k-j+3)]=j-1<j$. 
This means the $j$ rows are linearly dependent, a contradiction. 
So $c_{m-j+2}$ cannot divide all the numbers $x^{(i)}$, as stated.   \\

\noindent Proof of third step: We will prove for arbitrary $j$, that 
$\omega_{j}(X)$
is for $X\to\infty$ asymptotically bounded above by one of the fractions (depending on $\log(X)$) 
involved in the definition of $\wp_{j}(\eta)$, by which we mean that for any $\epsilon>0$ and $X=X(\epsilon)$
 large enough we have $\omega_{j}(X)<\wp_{j}(\eta)+\epsilon$.
Since $\omega_{j}=\limsup_{X\to\infty} \omega_{j}(X)$, $\epsilon\to 0$ shows the required result.\\
\noindent  So let $X$ be arbitrary but fixed and let $h$ be the index determined by $c_{h}\leq X< c_{h+1}$.
 We first prove that without loss of generality we may restrict to the case
where $X$ lies an interval of the shape $[c_{h},\frac{c_{h+1}}{4})$. \\
This is the case because the logarithm to the base $X=\frac{c_{h+1}}{4}$ of 

\[
 D_{\boldsymbol{x}}:= \max_{1\leq t\leq k} \vert \zeta_{t}x-y_{t}\vert, \qquad \boldsymbol{x}:=(x,y_{1},\ldots ,y_{k})
\]
for vectors $\boldsymbol{x}$ with $\vert x\vert \leq X=\frac{c_{h+1}}{4}$ 
 is asymptotically the same as to the base $c_{h+1}$. Indeed we have

\begin{eqnarray*}
 \log_{c_{h+1}}(D_{\boldsymbol{x}})&=& \frac{\log(D_{\boldsymbol{x}})}{\log(c_{h+1})}, \quad \log_{\frac{c_{h+1}}{4}}(D_{\boldsymbol{x}})= \frac{\log(D_{\boldsymbol{x}})}{\log(\frac{c_{h+1}}{4})} \\
 \lim_{h\to\infty}\frac{\log(c_{h+1})}{\log(\frac{c_{h+1}}{4})}&=&  \lim_{c_{h}\to\infty} \frac{\log(c_{h+1})}{\log(\frac{c_{h+1}}{4})}= \lim_{c_{h}\to\infty} \frac{\log(c_{h+1})}{\log(c_{h+1})-\log(4)}= 1,
\end{eqnarray*}
and hence

\begin{equation} \label{eq:zarteawa}
 \lim_{h\to\infty}\frac{\log_{c_{h+1}}(D_{\boldsymbol{x}})}{\log_{\frac{c_{h+1}}{4}}(D_{\boldsymbol{x}})}=1.
\end{equation}
On the other hand, since we can restrict to $\boldsymbol{x}$ belonging to some $\omega_{j}$ with
$j\geq 2$, all expressions $\log_{c_{h+1}}(D_{\boldsymbol{x}})$ are bounded above
by $2\omega_{2}\leq 2$ (see (\ref{eq:2222neu})) for $h$ sufficiently large. Together with (\ref{eq:zarteawa}) and
since this holds for every vector $\boldsymbol{x}$ for which $x\leq X$, the definition of the quantities $\omega_{j}$ immediately implies
that they remain unaffected by this change of base. \\

\noindent So let $h=h(X)$ be the index determined by $c_{h}\leq X< \frac{c_{h+1}}{4}$.
 By the second step of the proof (putting $m=h$) at least one of the $j$ linearly independent approximation vectors
 $\left(x,y_{1},y_{2},\ldots ,y_{k}\right)\in{\mathbb{Z}^{k+1}}$ has to satisfy the condition $c_{h-j+2}\nmid x$.
 Consider one of the $j$ approximation vectors with this property.
This means if we let $g-1$ be the largest index with $c_{g-1}\vert x$ as in step 1, we have $g-1\leq h-j+1$, i.e. $g\leq h-j+2$.
Further let $i$ be the index, for which $c_{h}=q_{N,i}$ belongs to the sequence $(q_{n,i})_{n\geq 1}$. At this point one should mention
that we will repeatedly
use step 1 in the following, neglecting the $o$-terms in the estimates 
(\ref{eq:mimena}),(\ref{eq:memina}),(\ref{eq:wilena}) as they do not affect the asymptotic behaviour we aim to prove. \\
\noindent First we treat the case $c_{g-1}\geq q_{N-1,i}$ (case 1 step 1). Note, that $\frac{c_{m-1}}{c_{m}}$ is monotonically
decreasing as $m$ increases by (\ref{eq:nurrrrr})-(\ref{eq:nurrr}) and (\ref{eq:18neu}), which we already used before.
Thus by $g\leq h-j+2$ and (\ref{eq:mimena}) we have

\begin{equation}  \label{eq:limeta}
 \max_{1\leq t\leq k}\vert x\zeta_{t}-y_{t}\vert\leq  \frac{c_{h-j+1}}{c_{h-j+2}}-o\left(\frac{c_{h-j+1}}{c_{h-j+2}}\right) \qquad \text{for}\quad h\to\infty.
\end{equation}
So $X\geq c_{h}$ implies

\[
 -\log_{X} \max_{1\leq t\leq k}\vert x\zeta_{t}-y_{t}\vert \leq -\frac{\log(\frac{c_{h-j+1}}{c_{h-j+2}})}{\log(c_{h})}=\frac{\log(c_{h-j+2})-\log(c_{h-j+1})}{\log(c_{h})}.  
\]
It is now easy to see by (\ref{eq:nurrrrr})-(\ref{eq:nurrr})
that for $h$ in a fixed residue class $\overline{h}$ of the residue system $\{1,2,\ldots ,k\}$ mod $k$, the  
right hand side tends to one of the fractions (depending on $\overline{h}$) in the definition of $\omega_{j}(\eta)$ 
or to zero as $h\to\infty$ or equivalently $X\to\infty$. 
Clearly, each expression in $\wp_{j}(\eta)$ is induced by some $\overline{h}$ in that way as well. 
This shows, that indeed we have $\omega_{j}(X)<\wp_{j}+\epsilon$ for any $\epsilon>0$ and $X=X(\epsilon)$ large enough. \\

\noindent In the remaining case $c_{g-1}<q_{N-1,i}$ (case 2 step 1) by (\ref{eq:memina}) and as $c_{h+1}$ in (\ref{eq:memina})
corresponds to $q_{N,i}$, 
$\max_{1\leq t\leq k}\vert x\zeta_{t}-y_{t}\vert$ is essentially bounded below by
$\frac{1}{q_{N-1,i+1}}$ (omitting the lower order terms and $i$ in the residue system $\{1,2,\ldots ,k\}$
mod $k$). We distinguish three cases now. \\
\noindent If we have $i\notin \{k-1,k\}$, approximation relative to base $X$ is bad, as in this case we have
$\lim_{N\to\infty} \frac{\log(q_{N-1,i+1})}{\log(q_{N,i})}=0$ as a consequence of (\ref{eq:nurrr}), so 
again by $X\geq c_{h}=q_{N,i}$, the expression

\[
 -\log_{X} \max_{1\leq t\leq k}\vert x\zeta_{t}-y_{t}\vert\leq -\log_{c_{h}} \max_{1\leq t\leq k}\vert x\zeta_{t}-y_{t}\vert\leq \frac{\log(q_{N-1,i+1})}{\log(c_{h})}=\frac{\log(q_{N-1,i+1})}{\log(q_{N,i})}
\]
tends to $0$ as $X\to\infty$, and we are done again.\\

\noindent In the case $i=k$, or equivalently $q_{N,k}\leq X< q_{N+1,1}$, due to
 $c_{g-1}<q_{N-1,i}=q_{N-1,k}$ we know again by (\ref{eq:memina}) that
$\max_{1\leq t\leq k}\vert x\zeta_{t}-y_{t}\vert\leq \frac{1}{q_{N,1}}$ and so $X\geq q_{N,k}$ implies

\[
 -\log_{X} \max_{1\leq t\leq k}\vert x\zeta_{t}-y_{t}\vert\leq -\log_{q_{N,k}} \max_{1\leq t\leq k}\vert x\zeta_{t}-y_{t}\vert\leq \frac{\log(q_{N,1})}{\log(q_{N,k})}.
\]
Hence by (\ref{eq:nurrr}),  for $h\to\infty$ we have the asyptotic

\[
 \frac{\log(q_{N,1})}{\log(q_{N,k})}\thicksim \frac{\log(q_{N,1})-\log(q_{N-1,k})}{\log(q_{N,k})}. 
\]
The right hand side, however, converges to $\eta_{1}=\wp_{k+1}(\eta)$ for $N\to\infty$ by (\ref{eq:nurrrrr}).
This shows that $\omega_{j}(X)\leq \wp_{k+1}(\eta)+\epsilon$ for any $\epsilon>0$ and $X=X(\epsilon)$ large enough
and together with $\wp_{k+1}\leq \wp_{j}(\eta)$ for $1\leq j\leq k+1$ we get
 $\omega_{j}\leq \wp_{j}(\eta)$ as desired.  \\

\noindent We devide the remaining case $i=k-1$, which means $q_{N,k-1}\leq X< q_{N,k}$, again into two cases.
If $x<\frac{1}{2}\frac{q_{N,k}}{q_{N-1,k}}$, it follows that 
$\max_{1\leq t\leq k}\vert x\zeta_{t}-y_{t}\vert$ is essentially bounded below by $\frac{1}{2}\frac{1}{q_{N-1,k}}$ in view
 of (\ref{eq:wilena}). This gives the estimate

\[
 -\log_{X} \max_{1\leq t\leq k}\vert x\zeta_{t}-y_{t}\vert\leq \frac{\log(\frac{1}{2}q_{N-1,k})}{\log(q_{N,k-1})}\leq  \frac{\log(q_{N-1,k})-\log(2)}{\log(q_{N,1})},
\]
which tends to $0$ for $X\to\infty$ by (\ref{eq:nurrr}).\\
\noindent Otherwise $x\geq \frac{1}{2}\frac{q_{N,k}}{q_{N-1,k}}$, which clearly implies 
$\lim_{N\to\infty}\frac{\log(x)}{\log(q_{N,k})}=1$ by (\ref{eq:nurrr}). In particular for every $\epsilon>0$ we have
 $\log(x)>(1+\epsilon)\log(q_{N,k})$ for
$N=N(\epsilon)$ sufficiently large.
 Note $X\geq x$ and that $X\to\infty$ is equivalent 
to $N\to\infty$. Combination of these facts together with the fact that 
$\max_{1\leq t\leq k}\vert x\zeta_{t}-y_{t}\vert$ is essentially bounded below by $\frac{1}{q_{N,1}}$ by (\ref{eq:memina}) 
yields the inequality

\[
 -\log_{X} \max_{1\leq t\leq k}\vert x\zeta_{t}-y_{t}\vert\leq \frac{\log(q_{N,1}+o(\log(q_{N,1})))}{\log(q_{N,k}(1+\epsilon))}
\]
for any $\epsilon>0$ and $N=N(\epsilon)$ sufficiently large.
However, the right hand side is of the form $\frac{\log(q_{N,1})}{\log(q_{N,k})}+o\left(\frac{\log(q_{N,1})}{\log(q_{N,k})}\right)$ 
as $N\to\infty$ and $\epsilon\to 0$,
 which tends to $\eta_{1}=\wp_{k+1}(\eta)$ for $N\to\infty$ as in the case $i=k$ above and so it is no improvement either. 
This shows step three.\\

\noindent Now it only remains to determine the approximation constants $\widehat{\omega}_{j}$. However, for $j\geq 2$ they are easily seen 
to be zero as a consequence of $\omega=\infty$. Indeed in this case we have $\underline{\psi}_{1}=-1$ by (\ref{eq:2})
and if for some $\epsilon>0$ we had $\overline{\psi}_{2}=\frac{1}{k}-\epsilon$, we would obtain 
$\sum_{j=1}^{k+1} \psi_{j}(q)=\psi_{1}(q)+\psi_{2}(q)+\sum_{j=3}^{k+1} \psi_{j}(q)\leq (-1+\frac{\epsilon}{3})+(\frac{1}{k}-\epsilon+\frac{\epsilon}{3})+(k-1)\frac{1}{k}<-\frac{\epsilon}{3}$
for a sequence of arbitrary large values $q$, a contradiction to (\ref{eq:zokz}).
So (\ref{eq:2}) again yields $\omega_{j}=0$ for $2\leq j\leq k+1$. \\
It remains to determine $\widehat{\omega}$. Let $X$ be a real number of the form $X=\lfloor \frac{c_{h+1}}{4}-1\rfloor$,
so that in particular we have $c_{h}\leq X<c_{h+1}$.\\ 
Putting $j=1$ in (\ref{eq:limeta}) and noting $X=\frac{c_{h+1}}{4}-1\geq \frac{c_{h+1}}{5}$ in 
the case $g\geq h+2-k$ we obtain 

\begin{eqnarray} \label{eq:harzharzl}
-\log_{X} \max_{1\leq t\leq k}\vert x\zeta_{t}-y_{t}\vert\leq -\log_{\frac{c_{h+1}}{5}}\max_{1\leq t\leq k}\vert x\zeta_{t}-y_{t}\vert\leq \frac{\log(c_{h+1})-\log(c_{h})}{\log(c_{h+1})-\log(5)}.  
\end{eqnarray}
If we now fix  
a residue class $\overline{h}$ for the values of $h$ in the residue system $\{1,2,\ldots ,k\}$ mod $k$, the right hand side of
(\ref{eq:harzharzl}) tends to one of the fractions in the definition of $\widehat{\wp}(\eta)$ as $h\to\infty$. In fact, as
$\overline{h}$ runs through the residue system $\{1,2,\ldots ,k\}$ mod $k$ this
induces a bijection between the residue classes $\overline{h}$ 
of the residue system $\{1,2,\ldots ,k\}$ mod $k$ and the expressions of $\widehat{\omega}$. \\
\noindent In the case $g\leq h+1-k$
as $\max_{1\leq t\leq k}\vert x\zeta_{t}-y_{t}\vert$ is essentially bounded below by
$\frac{1}{q_{N-1,i+1}}$ again, we have the upper esitmate

\[
-\log_{X} \max_{1\leq t\leq k}\vert x\zeta_{t}-y_{t}\vert\leq -\log_{\frac{c_{h+1}}{5}}\max_{1\leq t\leq k}\vert x\zeta_{t}-y_{t}\vert\leq \frac{\log(q_{N,h})}{\log(q_{N+1,h+1})-\log(5)}. 
\]
The right hand side, however, is smaller than the corresponding value in (\ref{eq:harzharzl}) for every $h$, so the
case $g\leq h+1-k$ does never give any improvement. 
 Thus, by its definition, the quantity $\widehat{\omega}$ can be estimated above by the minimum of the expressions
 of $\widehat{\wp}(\eta)$, which simply is $\widehat{\wp}(\eta)$.\\
\noindent On the other hand, fixing a residue classes $\overline{h}$ for $h$
in the residue system $\{1,2,\ldots ,k\}$ mod $k$ again and
 putting $x:=c_{h}$ and $y_{t}:=\lfloor \zeta_{t}c_{h}\rfloor$ for $1\leq t\leq k$, we obtain a bijection   
between the resulting values 

\[
\lim_{h\in{\overline{h}},h\to\infty}-\log_{X} \max_{1\leq t\leq k}\vert x\zeta_{t}-y_{t}\vert=\lim_{h\in{\overline{h}},X\to\infty}-\log_{X} \max_{1\leq t\leq k}\vert x\zeta_{t}-y_{t}\vert
\]
as $\overline{h}$ runs through $\{1,2,\ldots,k\}$ and the expressions involved in the definition of $\widehat{\wp}_{1}(\eta)$.
 Hence $\widehat{\omega}$ is at least as large as the minimum of these expressions, which again is $\widehat{\omega}$.
\end{proof}

\noindent Note, that in the special case $k=2$ we have $\eta_{1}+\eta_{2}=1$, and the approximation constants of $\boldsymbol{\zeta}$ in 
Theorem \ref{satz1} become

\begin{eqnarray*}
 \omega&=& \infty, \quad \omega_{2}=1, \quad  \omega_{3}=\eta_{1}    \\
 \widehat{\omega}&=&\eta_{2}, \quad \widehat{\omega}_{2}=0, \quad \widehat{\omega}_{3}=0.
\end{eqnarray*}
Particularly, we see that $\widehat{\omega}+\omega_{3}=1$, which is easily seen by straight forward computation with
repeated use of (\ref{eq:2})
to be equivalent to Jarnik's identity
in the form $\overline{\psi}_{1}+2\overline{\psi}_{1}\underline{\psi}_{3}+\underline{\psi}_{3}=0$, see Theorem 1.5 in \cite{5}.\\             

\noindent So far we have not asked for the numbers $\zeta_{1},\zeta_{2},\ldots ,\zeta_{k}$
 to be $\mathbb{Q}$-linearly independent together with $1$, which is the usual
assumption. For this purpose, we apply Theorem \ref{satz1} in the special case $\zeta_{j}=\sum_{n\geq 1}2^{-a_{n,j}}$
with suitable sequences $(a_{n,j})_{n\geq 1}$ for $1\leq j\leq k$.

\begin{coro}  \label{koroo}
For $1\leq j\leq k$ let $(a_{n,j})_{n\geq 1}$ be sequences with the properties

\begin{equation} \label{eq:frure}
 a_{1,1}<a_{1,2}< \ldots <a_{1,k}<a_{2,1}<a_{2,2}<\ldots a_{2,k}<a_{3,1}<\ldots
\end{equation}
and for $\eta\in{\mathbb{R}^{k}\times\overline{\mathbb{R}}}$ as in Theorem \ref{satz1} put

\begin{eqnarray}
 \lim_{n\to\infty} \frac{a_{n+1,1}-a_{n,k}}{a_{n+1,k}}&=& \eta_{1}   \label{eq:murrrrr}   \\  
 \lim_{n\to\infty} \frac{a_{n,i}-a_{n,i-1}}{a_{n,k}}&=& \eta_{i}, \quad 2\leq i\leq k. \\     \label{eq:murrrr}
 \lim_{n\to\infty} \frac{a_{n+1,1}}{a_{n,k}}&=& \eta_{k+1}=\infty  \label{eq:murrr}  
\end{eqnarray}
and $\zeta_{j}=\sum_{n\geq 1} 2^{-a_{n,j}}$ for $1\leq j\leq k$. Then the corresponding approximation constants are given
as in Theorem \ref{satz1}.
\end{coro}

\begin{proof}
 Clearly, if we put $q_{n,j}=2^{a_{n,j}}$, all conditions of Theorem \ref{satz1} are satisfied. 
\end{proof}
\noindent Now one can easily prove that there are uncountably many vectors $\boldsymbol{\zeta}\in{\mathbb{R}^{k}}$, such that additionally
 $1,\zeta_{1},\zeta_{2},\ldots ,\zeta_{k}$ are linearly independent over $\mathbb{Q}$. The arguments of the proof of the following
 Proposition \ref{prop1} are suitable to prove the existence of vectors $(\zeta_{1},\zeta_{2},\ldots ,\zeta_{k})$ 
for which $1,\zeta_{1},\zeta_{2},\ldots ,\zeta_{k}$ are linearly independent subject to certain approximation properties if this
existence was established without the linear independence condition.                                                                                                 

\begin{prop}  \label{prop1}
 One can choose sequences $(a_{n,j})_{n\geq 1}$ in Corollary \ref{koroo} such that $\{1,\zeta_{1},\ldots ,\zeta_{k}\}$ is linearly independent
over $\mathbb{Q}$.
\end{prop}

\begin{proof}
Note that in the case $k=1$ (\ref{eq:nurrrrr})-(\ref{eq:nurrr}) simply yield $\lim_{n\to\infty}\frac{a_{n+1}}{a_{n}}=\infty$
and it follows from Liouville's Theorem that the corresponding number of the form $\zeta=\sum_{n\geq 1} 2^{-a_{n}}$  
 is transcendental, in particular $\{1,\zeta\}$ is linearly 
independent over $\mathbb{Q}$. In the case $k\geq 2$ consider the numbers $a_{1,j}=j$ for $1\leq j\leq k$ and define 
the sequences $(a_{n,j})_{n\geq 1}$ by the recurrence relations

\begin{eqnarray}
a_{n+1,1} & = & \left\lfloor \frac{1}{\eta_{1}}n \cdot a_{n,k}(\eta_{1}) \right\rfloor        \label{eq:23neu}                       \\
a_{n+1,2} & = & \left\lfloor \frac{1}{\eta_{1}}n \cdot a_{n,k}(\eta_{1}+\eta_{2}) \right\rfloor    \label{eq:24neu}             \\
\vdots && \vdots  \\                                                                              \label{eq:25neu}
a_{n+1,k} & = & \left\lfloor \frac{1}{\eta_{1}}n \cdot a_{n,k}(\eta_{1}+\eta_{2}+\cdots +\eta_{k}) \right\rfloor. \label{eq:26neu}
\end{eqnarray}
One checks that (\ref{eq:23neu})-(\ref{eq:26neu}) imply (\ref{eq:nurrrrr})-(\ref{eq:nurrr}). Now we prove that we can change (\ref{eq:23neu})-(\ref{eq:26neu})
slightly such that $\{1,\zeta_{1},\ldots ,\zeta_{k}\}$ is linearly independent over $\mathbb{Q}$.\\
Let $(b_{n})_{n\geq 1}$ be the ordered combined set of all $a_{n,j}$ defined as above.
Note that we can obviously ''disturb'' the system (\ref{eq:23neu})-(\ref{eq:26neu}) a little by adding one to the elements of 
the form $b_{a}$ where $a\in{A}$ with $A$ an arbitrary subset of $\mathbb{N}$, without violating (\ref{eq:23neu})-(\ref{eq:26neu}).
 Noting that (\ref{eq:23neu})-(\ref{eq:26neu}) imply $\lim_{n\to\infty} \frac{a_{n+1,1}}{a_{n,1}}=\infty$, by the considerations of the case $k=1$ we know 
$\{1,\zeta_{1}\}$ is a $\mathbb{Q}$-linearly independent set, where 
$\zeta_{1}$ is generated by the sequence $(a_{n,1})_{n\geq 1}$ defined above, ie $\zeta_{1}=\sum_{n\geq 1} a_{n,1}^{-1}$.
If we now consider the set of all sequences
$\mathscr{A}_{2}=\{(a_{n,2}^{\prime})_{n\geq 1}\}$ which arise from $(a_{n,2})_{n\geq 1}$ as above by adding $1$ to 
$b_{a}$ with $a\in{A_{2}}$ for an arbitrary subset $A_{2}$ of $\mathbb{N}$, we see that $\mathscr{A}_{2}$ is uncountable. So 
there must be a number $\zeta_{2}^{\prime}$ generated by an $a_{n,2}^{\prime}\in{\mathscr{A}_{2}}$ with the property
that $\{1,\zeta_{1},\zeta_{2}^{\prime}\}$ are linearly independent over $\mathbb{Q}$, since the ${\mathbb{Q}}$-span 
of $\{1,\zeta_{1}\}$ is only countable. Now we proceed analoguously with sets $\mathscr{A}_{j}$ for $3\leq j\leq k$ and finally get a 
$\mathbb{Q}$-linearly independent set $\{1,\zeta_{1},\zeta_{2}^{\prime},\ldots ,\zeta_{k}^{\prime}\}$. As mentioned above, the 
 set $\{\zeta_{1},\zeta_{2}^{\prime},\ldots \zeta_{k}^{\prime}\}$ fulfills all the requirements.
\end{proof}
\noindent Note: Since algebraic numbers have countable cardinality one can readily generalize the proof above to show that we can even
ask $\boldsymbol{\zeta}$ to be algebraically independent. \\

\noindent We now give some applications of the above theorem. Note that (\ref{eq:mbni}),(\ref{eq:11}) and (\ref{eq:1000alk})
 imply that for all $\epsilon>0$ and sufficiently large $Q=Q(\epsilon)>0$ we have the bounds  

\begin{eqnarray*}
 -1&\leq& \psi_{1}(Q)\leq 0   \\
 \frac{j-k-1}{kj}-\epsilon &\leq& \psi_{j}(Q)\leq \frac{1}{k}, \qquad 2\leq j\leq k+1. 
\end{eqnarray*}
In the first Corollary we construct $\zeta_{1},\zeta_{2},\ldots ,\zeta_{k}$ for which
 each $\psi_{j}(Q)$ takes each of the values inside of the corresponding intervals 
$I_{1}:=(-1,0),I_{j}:=(\frac{j-k-1}{kj},\frac{1}{k})$
for arbitrarily large $Q$ simultaneously for all $1\leq j\leq k+1$. So roughly speaking in this case all $\psi_{j}$ take their possible 
range of values for arbitrarily large $(Q,\infty)$.
In particular the bounds in (\ref{eq:1000alk}) are best possible. 

\begin{coro}  \label{korol1}
 There exist $\zeta_{1},\zeta_{2},\ldots ,\zeta_{k}$ for which the set $\{1,\zeta_{1},\zeta_{2},\ldots ,\zeta_{k}\}$ is
$\mathbb{Q}$ linearly independent and such that

\begin{eqnarray*}
 \omega_{j}&=&\frac{1}{j-1} \qquad 1\leq j\leq k+1   \\
 \widehat{\omega}&=& \frac{1}{k}                               \\
 \widehat{\omega}_{j}&=& 0 \qquad 2\leq j\leq k+1.
\end{eqnarray*}
  
\end{coro}

\begin{proof}
Note that by means of proposition \ref{prop1} for every $\eta\in{\mathbb{R}^{k+1}}$ subject to the restrictions of Theorem \ref{satz1}
 we can construct $\boldsymbol{\zeta}=(\zeta_{1},\zeta_{2},\ldots ,\zeta_{k})$ together with $1$ linearly independent over 
$\mathbb{Q}$ such that Theorem \ref{satz1} holds. 
 Putting $\eta_{2}=\eta_{3}=\cdots =\eta_{k+1}=\frac{1}{k}$ in Theorem \ref{satz1} immediately gives all the stated equalities
 for this $\boldsymbol{\zeta}$.  
\end{proof}

\noindent Now we want to give our first explicit construction of special cases of Schmidt's conjecture, which was proved by Moshchevitin
 in a nonconstructive way in \cite{3}. The conjecture states, that for each integer pair $(k,i)$ with $k\geq 2,1\leq i\leq k-1$ there                   
 exists a vector 
$\boldsymbol{\zeta}\in{\mathbb{R}^{k}}$ with $\{1,\boldsymbol{\zeta}\}$ linearly independent over $\mathbb{Q}$
such that $\lim_{q\to\infty}\lambda_{i}(q)=0$ and $\lim_{q\to\infty} \lambda_{i+2}=\infty$.
Note that we cannot have $\lim_{q\to\infty}\lambda_{i}(q)=0,\lim_{q\to\infty} \lambda_{i+1}(q)=\infty$ for any $i$
because of the assumption of linear independence because of (\ref{eq:asdf}), see also the introduction in \cite{3}. We now give a generalisation          
of this fact in the special case $i=1$ for arbitrary $k\geq 2$.

\begin{coro}
 Let $k\geq 2$ and $3\leq r\leq k+1$ be integers. Then there exists $\boldsymbol{\zeta}\in{\mathbb{R}^{k}}$ with 
$\{1,\boldsymbol{\zeta}\}$ linearly independent over $\mathbb{Q}$ such that 
\begin{eqnarray*}
 \overline{\psi}_{1}&<&0     \\
 \underline{\psi}_{j}&<&0< \:\: \overline{\psi}_{j}, \qquad   2\leq j \leq r-1,  \\  
 \underline{\psi}_{j}&>&0, \qquad  r\leq j\leq k+1. 
\end{eqnarray*}
The case $r=3$ clearly implies Schmidt's conjecture for $i=r-2=1$.
\end{coro}

\begin{proof}
We may assume $k\geq 3$, because for $k=2$ the Corollary only states that $\overline{\psi}_{1}<0$ and $\underline{\psi}_{3}>0$
is possible, which only requires $\overline{\psi}_{1}<0$ by (\ref{eq:zokz}) and for any choice of
 $\eta=(\eta_{1},\eta_{2})\neq (1/2,1/2)$ the construction of Theorem \ref{satz1} gives an example. 
We apply Theorem \ref{satz1} with $\eta$ defined by

\begin{eqnarray}
 \eta_{1}&=& \frac{\alpha^{k-1}}{1+\alpha+\cdots +\alpha^{k-1}},   \label{eq:b1} \\
 \frac{\eta_{j}}{\eta_{j+1}}&=& \alpha, \qquad 1\leq j\leq k-1.    \label{eq:b2}
\end{eqnarray}
The parameter $\alpha\in{\{0,1\}}$ will be chosen later in dependence of $(r,k)$. 
 First note that by (\ref{eq:2}) our inequalities translate to

\begin{eqnarray}
 \widehat{\omega}&>& \frac{1}{k},  \label{eq:a1}  \\
 \widehat{\omega}_{j}&<&\frac{1}{k}< \:\: \omega_{j}, \qquad 2\leq j\leq r-1,  \label{eq:a2}  \\
 \omega_{j}&<&\frac{1}{k},   \qquad r\leq j\leq k+1.   \label{eq:a3} 
\end{eqnarray}
\noindent Now note, that the left hand side of (\ref{eq:a2}) trivially holds by Theorem \ref{satz1}. With (\ref{eq:b1}),(\ref{eq:b2})
the quantity $\widehat{\wp}(\eta)$ of Theorem \ref{satz1} becomes

\[
  \widehat{\omega}= \min \left\{ 1, \frac{1}{1+\alpha}, \ldots ,\frac{1}{1+\alpha+\cdots + \alpha^{k-1}}\right\}= \frac{1}{1+\alpha+\cdots + \alpha^{k-1}}.
\]
Moreover the assumption $\alpha<1$ implies that (\ref{eq:a1}) holds. To obtain the remaining inequalities it is obviously sufficient
to prove $\omega_{r}<\frac{1}{k}<\omega_{r-1}$ for some $\alpha$, which
 by (\ref{eq:b1}),(\ref{eq:b2}) and Theorem \ref{satz1} is equivalent to

\begin{equation} \label{eq:300}
\frac{\alpha^{r-2}}{1+\alpha+\cdots +\alpha^{r-2}}< \frac{1}{k}< \frac{\alpha^{r-3}}{1+\alpha+\cdots +\alpha^{r-3}},  
\end{equation}
 
\noindent since the last term in $\wp_{j}(\eta)$ is easily seen to be the largest in $\wp_{j}$ in our special case of constant ratios.
For $r=3$ this reduces to $\frac{\alpha}{\alpha+1}<\frac{1}{k}$, which is obviously true if we choose any 
$\alpha\in{(0,\frac{1}{k})}$, so we can assume $r\geq 4$. Defining the functions

\[
\Phi_{u}:\quad  \alpha \longmapsto \frac{\alpha^{u}}{1+\alpha+\cdots +\alpha^{u}}
\]
shows that (\ref{eq:300}) in the cases left to consider is equivalent to $\phi_{u+1}(\alpha)<\frac{1}{k}<\phi_{u}(\alpha)$ 
for $1\leq u\leq k-2$. It is easy to check that all these $\phi_{u}$ are continuous, $\phi_{u}(\alpha)>\phi_{u+1}(\alpha)$ and
$\phi_{u}(0)=0, \phi_{u}(1)=\frac{1}{u+1}>\frac{1}{k}$. Further more from

\begin{eqnarray*}
 \Phi_{u}^{\prime}(\alpha) & = & \frac {u\alpha^{u-1}(1+\alpha+\cdots +\alpha^{u})-\alpha^{u}(1+2\alpha+\cdots +u\alpha^{u-1})}{(1+\alpha+\cdots +\alpha^{u})^{2} } \\
& = & \frac {\alpha^{2u-2}+2\alpha^{2u-3}+\cdots +(u-1)\alpha^{u}+u\alpha^{u-1}}{(1+\alpha+\cdots +\alpha^{u})^{2}} > 0
\end{eqnarray*}
we deduce that they are monotonically increasing in $\alpha$. Combination of these properties implies that for fixed $u$
there exists some $t\in{(0,1)}$
 such that $\phi_{u}(t)=\frac{1}{k}$ by intermediate value theorem. It further follows from these considerations
 on the one hand $\phi_{u}(\alpha)>\frac{1}{k}$ for $\alpha>t$,
and on the other hand the existence of an interval $\alpha\in{(t_{0},t_{1})}$ with $t_{0}<t<t_{1}$ such
 that $\phi_{u+1}(\alpha)<\frac{1}{k}$.
Thus, for all $\alpha\in{(t,t_{1})}$ we have $\phi_{u+1}(\alpha)<\frac{1}{k}<\phi_{u}(\alpha)$.
\end{proof}

\subsection{The case $\omega<\infty$}

We now aim to give similar results for vectors $\boldsymbol{\zeta}=(\zeta_{1},\zeta_{2},\ldots ,\zeta_{k})$ whose components 
$\zeta_{j}$ have one-dimensional approximation constant $\omega<\infty$. Hence in particular the simultaneous aproximation
constant $\omega$ is finite too, as by definition it cannot exceed the minimum of the one-dimensional constants. 
As in Theorem \ref{satz1} each $\zeta_{j}$ will
 be the sum of the re ciprocals of integers $q_{n,j}$ that satisfy (\ref{eq:qq}),(\ref{eq:erwin}), 
but the conditions (\ref{eq:nurrrrr})-(\ref{eq:17neu}) will be altered in ways to obtain a symmetric situation in all $\zeta_{j}$, which
will be more convenient for the purposes of chapter 2. 
We start with proving the following

\begin{theo}  \label{satz2}
 For $1\leq j\leq k$ let $\zeta_{j}=\sum_{n\geq 1} \frac{1}{q_{n,j}}$ where
 $(q_{n,j})_{n\geq 1}$ are sequences of integers for which (\ref{eq:qq}),(\ref{eq:erwin}) are satisfied and that for

\[
  (b_{n})_{n\geq 1}= (\log(q_{1,1}),\log(q_{2,1}),\ldots ,\log(q_{k,1}),\log(q_{1,2}),\ldots)
\]
the inequality

\begin{eqnarray*}
 \liminf_{n\to\infty} \frac{b_{n+1}}{b_{n}}>2
\end{eqnarray*}
is satisfied. Then the first $(k-1)$ approximation constants relative to $\zeta_{1},\zeta_{2},\ldots ,\zeta_{k}$ are given by

\begin{eqnarray*}
 \omega&=& \limsup_{n\to \infty} \frac{b_{n+1}-b_{n}}{b_{n}}, \qquad  \quad \quad \quad \widehat{\omega}= \liminf_{n\to\infty} \frac{b_{n}-b_{n-1}}{b_{n}}  \\
 \omega_{2}&=& \limsup_{n\to\infty} \frac{b_{n}-b_{n-1}}{b_{n}}, \qquad \quad \quad \quad \widehat{\omega}_{2}= \liminf_{n\to\infty} \frac{b_{n-1}-b_{n-2}}{b_{n}}  \\
 \omega_{3}&=& \limsup_{n\to\infty} \frac{b_{n-1}-b_{n-2}}{b_{n}}, \qquad \quad \quad \widehat{\omega}_{3}= \liminf_{n\to\infty} \frac{b_{n-2}-b_{n-3}}{b_{n}}  \\
 &\vdots& \quad \vdots \qquad \vdots \qquad \qquad \qquad \qquad \qquad \vdots \qquad \vdots \qquad \vdots  \\
 \omega_{k-1}&=& \limsup_{n\to\infty} \frac{b_{n-k+3}-b_{n-k+2}}{b_{n}}, \quad \widehat{\omega}_{k-1}= \liminf_{n\to\infty} \frac{b_{n-k+2}-b_{n-k+1}}{b_{n}}.
\end{eqnarray*}
Further more we have the inequlities

\begin{eqnarray*}
\omega_{k}&\geq& \limsup_{n\to\infty} \frac{b_{n-k+2}-b_{n-k+1}}{b_{n}}, \quad \widehat{\omega}_{k}\geq \liminf_{n\to\infty} \frac{b_{n-k+1}-b_{n-k}}{b_{n}} \\
\omega_{k+1}&\geq& \limsup_{n\to\infty} \frac{b_{n-k+1}-b_{n-k}}{b_{n}}, \quad \widehat{\omega}_{k+1}\geq \liminf_{n\to\infty} \frac{b_{n-k}-b_{n-k-1}}{b_{n}}.
\end{eqnarray*}
\end{theo}

\begin{proof}

Denote the right hand side expressions by $\wp_{j}$ respectively $\widehat{\wp}_{j}$.
 We prove the inequalities $\omega_{j}\geq \wp_{j}, \widehat{\omega}_{j}\geq \widehat{\wp}_{j}$ for $1\leq j\leq k+1$ and then
$\omega_{j}\leq \wp_{j}, \widehat{\omega}_{j}\leq \widehat{\wp}_{j}$ for $1\leq j\leq k-1$. This obviously yields the assertions of the 
Theorem. \\
\noindent Throughout the proof, for arbitrary $X$ let the integer $h$ be determined as the index for which $c_{h}\leq X<c_{h+1}$.
Obviously $X\to\infty$ is equivalent to $h\to\infty$, which will often be used implicitly. We first prove $\omega_{j}\geq \wp_{j}$
for $1\leq j\leq k+1$.\\
\noindent Assume $j$ arbitrary but fixed.
Observe that for every $h$ if we put
$x^{(i)}=c_{h+1-i}$, $y_{t}^{(i)}= \lfloor \zeta_{t}x^{(i)}\rfloor$ for $1\leq t\leq k$ and $1\leq i\leq j$, we obtain

\begin{equation} \label{eq:tiryl}
 \max_{1\leq i\leq j}\max_{1\leq t\leq k} \left\vert \zeta_{t}x^{(j)}-y_{t}^{(j)}\right\vert= \left\vert \zeta_{\overline{t_{0}}}x^{(j)}-y_{t_{0}}^{(j)}\right\vert =  \frac{c_{h+1-j}}{c_{h+2-j}}+o\left(\frac{c_{h+1-j}}{c_{h+2-j}}\right) \qquad \text{as} \quad h\to\infty
\end{equation}
with $t_{0}:=h+2-j$, as explained in the proof of Theorem \ref{satz1}. Choose an integer sequence 
$(h_{r})_{r\geq 1}$ of values for $h$ in (\ref{eq:tiryl}) such that with
$(X_{r})_{r\geq 1}:=(x_{r})_{r\geq 1}:=(c_{h_{r}})_{r\geq 1}$ we have

\begin{equation} \label{eq:mtiryl}
 \lim_{r\to\infty} -\log_{X_{r}}\frac{c_{h_{r}+1-j}}{c_{h_{r}+2-j}}=\limsup_{h\to\infty} -\log_{c_{h}}\frac{c_{h+1-j}}{c_{h+2-j}}=\limsup_{h\to\infty}\frac{\log(c_{h+2-j})-\log(c_{h+1-j})}{\log(c_{h})},
\end{equation}
which is possible by definition of the $\limsup$. Leaving out the lineary independence condition, the so constructed vectors
$(x^{(i)},y_{1}^{(i)},\ldots ,y_{k}^{(i)}), 1\leq i\leq j$, lead to the values $\wp_{j}$ as lower
 bounds for $\omega_{j}$ in view of (\ref{eq:tiryl}).
However, the missing lineary independence condition is now obtained exactly as in Theorem \ref{satz1}.\\  

\noindent To prove $\widehat{\omega}_{j}\geq \widehat{\wp}_{j}$,
 consider any sequence $(h_{r})_{r\geq 1}$ and the same approximation vectors as in the
proof of $\omega_{j}\geq \wp_{j}$ but take
logarithms to base $c_{h+1}>X$ instead of base $c_{h}$.
This yields a lower estimate for $\omega_{j}(X)$ for $X\in{[c_{h_{r}},c_{h_{r+1}})}$.
As this is valid for any sequence $(h_{r})_{r\geq 1}$ we obtain the lower bounds $\widehat{\wp}_{j}=\liminf_{X\to\infty} \omega_{j}(X)$
 for $\widehat{\omega}_{j}$ by definition of $\liminf$, as we claimed. 
So far we have established the lower bounds
$\wp_{j}$ (respectively $\widehat{\wp}_{j}$) for $\omega_{j}$ (respectively $\widehat{\omega}_{j}$) for $1\leq j\leq k+1$.\\
\noindent For the upper bounds note first that with basically the same arguments as in the proof of
step 3 in Theorem \ref{satz1} we can restrict to 
the case $c_{h}\leq X\leq \frac{1}{4}c_{h+1}$. Further step 1 and step 2 of $\wp_{j}(\eta)\leq \omega_{j}$ in the proof
 of Theorem \ref{satz1} remain valid in the present situation.
Indeed, the estimates (\ref{eq:mimena}),(\ref{eq:memina}) are already valid under the assumtion $\frac{b_{n+1}}{b_{n}}>2$, which
is weaker than (\ref{eq:nurrr}) used in Theorem \ref{satz1}. The proof of step 2 is analoguous.\\
\noindent Now for every fixed $X$
we divide all approximation vectors $(x,y_{1},\ldots ,y_{k})$ with $x\leq X$ into two categories.
Let $g$ be the largest integer such that $c_{g-1}\vert x$ for an approximation vector $(x,y_{1},\ldots ,y_{k})$
 as in step 1 of Theorem \ref{satz1}. The distinction of vectors with $g>h+1-k$, which we will call vectors of category 1,
 and $g\leq h+1-k$, which we will call vectors of category 2, now leads to 2 cases. \\
\noindent Case 1: If for fixed $X$ with $c_{h}\leq X<\frac{1}{4}c_{h+1}$ we have $g> h+1-k$ for an approximation vector
(i.e. it belongs to category 1), (\ref{eq:mimena}) 
implies that 

\begin{equation} \label{eq:manitur}
 \omega_{j}(X) \leq -\log_{c_{h}} \frac{c_{h+1-j}}{c_{h+2-j}}= \frac{\log(c_{h+2-j})-\log(c_{h+1-j})}{\log(c_{h})}, \qquad 1\leq j\leq k+1.
\end{equation}
So we have that for every $X$ the quantity $\wp_{j}$ is an upper bound for
$\omega_{j}^{1}(X)$, by which we mean the supremum over all real numbers $\nu$ such that
 (\ref{eq:1}) has $j$ linearly independent vector solutions all of which are of the first category.
Hence if we define $\omega_{j}^{1}=\limsup_{X\to\infty}\omega_{j}^{1}(X)$, we get 

\begin{equation} \label{eq:yaqmko}
 \omega_{j}^{1}\leq \wp_{j}, \qquad 1\leq j\leq k+1. 
\end{equation}
In order to give a connection between the approximation constants $\widehat{\omega}_{j}$ and approximation
vectors of category 1, we define $\widehat{\omega}_{j}^{1}:=\liminf_{X\to\infty} \omega_{j}^{1}(X)$.
We start with an arbitrary sequence $(X_{r})_{r\geq 1}$ with corresponding subsequence $(c_{h_{r}})_{r\geq 1}$ of $(c_{h})_{h\geq 1}$
and define a sequence $(X_{r}^{\prime})_{r\geq 1}$ by putting $X_{r}^{\prime}:=\frac{1}{5}c_{h_{r}+1}$. 
In view of (\ref{eq:mimena}) and observing that the fractions $\frac{c_{m}}{c_{m+1}}$ are monotonically decreasing by our assumption
$\frac{b_{n+1}}{b_{n}}>2$, 
we get the upper estimate for the approximation constants $\omega_{j}^{1}(X_{r}^{\prime})$

\begin{equation} \label{eq:tomoy}
\omega_{j}^{1}(X_{r}^{\prime})\leq -\log_{\frac{1}{5}X_{r+1}} \left(\frac{c_{h_{r}+2-j}}{c_{h_{r}+1-j}}\right)= \frac{\log(c_{h_{r}+2-j})-\log(c_{h_{r}+1-j})}{\log(c_{h_{r}+1})-\log(5)}
\end{equation}
for $1\leq j\leq k+1$.
If we specify a sequence $(X_{r})_{r\geq 1}$ for which the corresponding sequence $(h_{r})_{r\geq 1}$ has the property

\begin{equation} \label{eq:albla}
 \lim_{r\to\infty} -\log_{h_{r}}\frac{c_{h_{r}+1-j}}{c_{h_{r}+2-j}}=\liminf_{h\to\infty} -\log_{c_{h}}\frac{c_{h+1-j}}{c_{h+2-j}}=\liminf_{h\to\infty}\frac{\log(c_{h+2-j})-\log(c_{h+1-j})}{\log(c_{h})},
\end{equation}
which is possible again by definition of $\liminf$, we put $n=h+1$ in the definition of $\widehat{\wp}_{j}$ so that
the right hand side of (\ref{eq:tomoy}) tends to $\widehat{\wp}_{j}$ as $r\to\infty$.
Thus $\lim_{r\to\infty}\omega_{j}^{1}(X_{\underline{r}})$ exists and is bounded above by $\widehat{\wp}_{j}$.
In particular 

\begin{equation} \label{eq:yaqokm}
\widehat{\omega}_{j}^{1}\leq \widehat{\wp}_{j}, \qquad 1\leq j\leq k+1. 
\end{equation}
Case 2: In the other case $g\leq h+1-k$ (i.e. the vector belongs to category 2), 
consider first an arbitrary sequence $(X_{r})_{r\geq 1}$ that tends monotonically to infinity
and the corresponding subsequence $(c_{h_{r}})_{r\geq 1}$ of $(c_{h})_{h\geq 1}$ determined by $c_{h_{r}}\leq X_{r}< c_{h_{r+1}}$.
Recall that without loss of generality we can assume
$c_{h_{r}}\leq X_{r}<\frac{1}{4}c_{h_{r}+1}$, as this does not affect the approximation constants.
By (\ref{eq:memina}) we have

\begin{equation}  \label{eq:huhnfisch}
 \max_{1\leq t\leq k} \vert \zeta_{t}x-y_{t}\vert \geq \frac{1}{c_{h_{r}+2-k}}-o\left(\frac{1}{c_{h_{r}+2-k}}\right).
\end{equation}
Furthermore define $\omega_{j}^{2}(X),\omega_{j}^{2}$ as the supremum of all $\nu$, such that (\ref{eq:1}) has $j$ linearly independent
vector solutions with at least one vector of category 2, for every fixed $X>0$. 
Taking logarithms to the bases $c_{h_{r}}<X_{r}$ for $r\to\infty$,  (\ref{eq:huhnfisch}) implies that for $1\leq j\leq k+1$ 
the expression $\omega_{j}^{2}(X_{r})$ is bounded above by

\begin{equation}
 -\log_{X_{r}}\left(\frac{c_{h_{r}+2-k}}{c_{h_{r}+3-k}}\right)\leq -\log_{c_{h_{r}}}\left(\frac{c_{h_{r}+2-k}}{c_{h_{r}+3-k}}\right)=\frac{\log(c_{h_{r}+3-k})-\log(c_{h_{r}+2-k})}{\log(c_{h_{r}})} \label{eq:qply}.  
\end{equation}
Note that the bound in (\ref{eq:qply}) is valid for any sequence
 $(X_{r})_{r\geq 1}$
and the corresponding sequence $(c_{h_{r}})_{r\geq 1}$. Hence on the one hand we have

\begin{equation} \label{eq:malina}
 \mathscr{A}:= \limsup_{h\to\infty} \frac{\log(c_{h+3-k})-\log(c_{h+2-k})}{\log(c_{h})}\geq \limsup_{X\to\infty}\omega_{j}^{2}(X)=\omega_{j}^{2}
\end{equation}
simply by the definition of $\limsup$. Observe $\mathscr{A}=\wp_{k-1}$ (where $n$ in the 
definition of $\wp_{k-1}$ corresponds to $h$ in the definition of $\mathscr{A}$) and hence 

\begin{equation}  \label{eq:yaqomk}
\omega_{j}^{2}\leq \wp_{k-1}, \qquad 1\leq j\leq k+1.
\end{equation}
\noindent On the other hand, for any sequence $(X_{r})_{r\geq 1}$ by (\ref{eq:huhnfisch}) and our assumption
 $\frac{b_{n+1}}{b_{n}}=\frac{\log(c_{n+1})}{\log(c_{n})}>2$, for $r$ sufficiently large and $1\leq j\leq k-1$ we have

\begin{equation} \label{eq:ronzor}
 \omega_{j}^{2}(X_{r})\leq -\log_{X_{r}}\left(\frac{1}{c_{h_{r}+2-k}}-o\left(\frac{1}{c_{h_{r}+2-k}}\right)\right)< -\log_{X_{r}}\left(\frac{c_{h_{r}+2-k}}{c_{h_{r}+3-k}}\right)\leq \omega_{k-1}^{1}(X_{r}),
\end{equation}
where the right inequality is a consequence of the construcions $\omega_{j}\geq \wp_{j}, \widehat{\omega}_{j}\geq \widehat{\wp}_{j}$
in the case $j=k-1$ (see (\ref{eq:tiryl}),(\ref{eq:mtiryl})).
Since this holds for any sequence $(X_{r})_{r\geq 1}$, we have  

\begin{equation} \label{eq:schmetzeling}
 \widehat{\omega}_{j}^{2}:=\liminf_{X\to\infty} \omega_{j}^{2}(X)\leq \liminf_{X\to\infty} \omega_{k-1}^{1}(X)= \widehat{\omega}_{k-1}^{1}\leq \widehat{\omega}_{j}^{1}, \quad 1\leq j\leq k-1.
\end{equation}
\noindent We now combine the results above for the quantities $\omega_{j}^{s},\widehat{\omega}_{j}^{s}$, $s\in{\{1,2\}}$
 to derive the required upper bounds.
As every approximation vector is either of category 1 or category 2 for fixed $X>0$,
the defintions of $\omega_{j}^{s}, s\in{\{1,2\}}$, imply $\omega_{j}(X)=\max\{\omega_{j}^{1}(X),\omega_{j}^{2}(X)\}$ for every $X>0$. 
Observe that for any functions $f,g:\mathbb{R}^{+}\mapsto \mathbb{R}$ we have

\begin{equation*}
 \max\left\{\limsup_{X\to\infty} f(X),\limsup_{X\to\infty} g(X)\right\}= \limsup_{X\to\infty} \max\left\{f(X),g(X)\right\}  
\end{equation*}
Applying this on $f(X)=\omega_{j}^{1}(X),g(X)=\omega_{j}^{2}(X)$ implies that
$\omega_{j}$, which is by definition $\limsup_{X\to\infty}\omega_{j}(X)$, equals the maximum of 
$\omega_{j}^{1}=\limsup_{X\to\infty}\omega_{j}^{1}(X)$ and $\omega_{j}^{2}=\limsup_{X\to\infty}\omega_{j}^{2}(X)$.
By (\ref{eq:yaqmko}),(\ref{eq:yaqomk}) for $1\leq j\leq k-1$ this maximum is
$\max\{\wp_{j},\wp_{k-1}\}=\wp_{j}$, which proves the upper bounds for $\omega_{j}$, $1\leq j\leq k-1$. \\
\noindent It remains to check the upper estimates for the constants $\widehat{\omega}_{j}$.
In view of $\omega_{j}(X)=\max\{\omega_{j}^{1}(X),\omega_{j}^{2}(X)\}$, (\ref{eq:yaqokm}) and (\ref{eq:schmetzeling}),
we obtain 

\begin{equation*} \label{eq:camillard}
\widehat{\omega}_{j}=\liminf_{X\to\infty} \max_{s=1,2}\{\omega_{j}^{s}(X)\}= \liminf_{X\to\infty} \omega_{j}^{1}(X)= \widehat{\omega}_{j}^{1}\leq \widehat{\wp}_{j}
\end{equation*}
for $1\leq j\leq k-1$.
\end{proof}

\noindent Remark: It is rather clear from the proof that Theorem \ref{satz2} remains valid for $C=\infty$ too. We will need this later
 in Theorem \ref{theo9}.

\begin{coro}  \label{korolar}
Let the assumptions of Theorem \ref{satz2} be satisfied and assume further the existence of the limit
 of the quotients $\frac{b_{n+1}}{b_{n}}$, i.e.

\[
 \lim_{n\to\infty} \frac{b_{n+1}}{b_{n}}=:C\geq 2.
\]
Then the first $(k-1)$ approximation constants are given by
 \begin{eqnarray*}
 \omega \quad = &C-1&   \\
  \omega_{2} \quad = &\frac{C-1}{C}&=  \quad \widehat{\omega}             \\
 \omega_{3} \quad = &\frac{C-1}{C^{2}}&=  \quad \widehat{\omega}_{2}     \\
 \vdots \qquad &\vdots& \quad \vdots                      \\
 \omega_{k-1} \quad= &\frac{C-1}{C^{k-2}}&=  \quad \widehat{\omega}_{k-2}   \\
 &\frac{C-1}{C^{k-1}}&=  \quad \widehat{\omega}_{k-1}.
\end{eqnarray*}
For the remaining approximation constants we have the inequlities

\begin{eqnarray*}
\omega_{k}&\geq& \frac{C-1}{C^{k-1}}   \\
\omega_{k+1} \quad \geq \quad \widehat{\omega}_{k}&\geq& \frac{C-1}{C^{k}}     \\
 \widehat{\omega}_{k+1}&\geq& \frac{C-1}{C^{k+1}}.
\end{eqnarray*}
\end{coro}

\begin{proof}
For every $h\geq 1$ we have

\[
 \lim_{n\to\infty} \frac{b_{n+h}}{b_{n}}= \lim_{n\to\infty} \frac{b_{n+h}}{b_{n+h-1}}\frac{b_{n+h-1}}{b_{n+h-2}}\cdots \frac{b_{n+1}}{b_{n}}=C^{h}
\]
 and hence Theorem \ref{satz2} yields the claimed result.
\end{proof}
\noindent Remark: The bounds for $\omega_{k},\widehat{\omega}_{k},\omega_{k+1},\widehat{\omega}_{k+1}$ could be improved further to

\begin{eqnarray*}
 \frac{C-1}{C^{k-1}} \quad \leq \quad \omega_{k} &\leq& \max\left\{ \frac{C}{C^{k}-1},\frac{C-1}{C^{k-1}}\right\}   \\
 \widehat{\omega}_{k}&= &\frac{C-1}{C^{k}-1}   \\
 \omega_{k+1}&=& \frac{1}{C^{k-1}}             \\
 \min\left\{\frac{1}{C^{k}-1},\frac{C-1}{C^{k}}\right\} \quad \leq \quad \widehat{\omega}_{k+1} &\leq& \frac{1}{C^{k}-1}     
\end{eqnarray*}
by a rather long and technical proof that we will not present here. In particular in the case $C\geq \beta_{k}>2$, where $\beta_{k}$
is the largest real root of $P_{k}(x)=x^{k+1}-2x^{k}-x+1$, we have

\begin{eqnarray*}
 \omega_{k}&=& \frac{C-1}{C^{k-1}}    \\
 \widehat{\omega}_{k}&= &\frac{C-1}{C^{k}-1}   \\
 \omega_{k+1}&=& \frac{1}{C^{k-1}}              \\
 \widehat{\omega}_{k+1}&=& \frac{1}{C^{k}-1}.
\end{eqnarray*}
Let us call the assumptions of Theorem \ref{satz2} without the growth condition of $\frac{b_{n+1}}{b_{n}}$ the {\em basic assumptions}
of Theorem \ref{satz2} in the sequel.
We can generelize the idea of the proof of Theorem \ref{satz2} to get

\begin{theo} \label{satz6}
Given the basic assumptions of Theorem \ref{satz2}, we consider some fixed $d\in\{1,2,\ldots,k-1\}$ 
and define $\kappa_{d}$ to be the largest real root 
of $P_{d}(x):=x^{d}-x^{d-1}-1$. Then if 

\begin{enumerate}
 \item $\frac{b_{n+1}}{b_{n}}>\kappa_{d},\qquad \text{for all} \quad n\geq 1$,  \\
\noindent \item the sequence $(d_{n})_{n\geq 1}:=(b_{n+1}-b_{n})_{n\geq 1}$ is monotonically increasing \\
\end{enumerate}
\noindent are satisfied, the first $(k-d)$ approximation constants are given by

 \begin{eqnarray*}
 \omega&=& \limsup_{n\to \infty} \frac{b_{n+1}-b_{n}}{b_{n}}, \qquad \qquad \quad \quad \quad \widehat{\omega}= \liminf_{n\to\infty} \frac{b_{n}-b_{n-1}}{b_{n}}  \\
 \omega_{2}&=& \limsup_{n\to\infty} \frac{b_{n}-b_{n-1}}{b_{n}}, \qquad \qquad \quad \quad \quad \widehat{\omega}_{2}= \liminf_{n\to\infty} \frac{b_{n-1}-b_{n-2}}{b_{n}}  \\
 \omega_{3}&=& \limsup_{n\to\infty} \frac{b_{n-1}-b_{n-2}}{b_{n}}, \qquad \qquad \quad \quad \widehat{\omega}_{3}= \liminf_{n\to\infty} \frac{b_{n-2}-b_{n-3}}{b_{n}}  \\
 &\vdots& \quad \vdots \qquad \vdots \qquad \qquad \qquad \qquad \qquad \qquad \quad \vdots\qquad \vdots \qquad \vdots \\
 \omega_{k-d}&=& \limsup_{n\to\infty} \frac{b_{n-k+d+2}-b_{n-k+d+1}}{b_{n}}, \quad \widehat{\omega}_{k-d}= \liminf_{n\to\infty} \frac{b_{n-k+d+1}-b_{n-k+d}}{b_{n}}.
\end{eqnarray*}
Furthermore we have the inequalities 

\begin{eqnarray*}
 \omega_{k-d+1}&\geq& \limsup_{n\to\infty} \frac{b_{n-k+d+1}-b_{n-k+d}}{b_{n}}, \quad \widehat{\omega}_{k-d+1}\geq \liminf_{n\to\infty} \frac{b_{n-k+d+1}-b_{n-k+d}}{b_{n}}\\
&\vdots& \quad \vdots \qquad \vdots \qquad \qquad \qquad \qquad \qquad \qquad \vdots \qquad \vdots \qquad \vdots  \\
\omega_{k+1}&\geq& \limsup_{n\to\infty} \frac{b_{n-k+1}-b_{n-k}}{b_{n}}, \qquad \quad \quad \widehat{\omega}_{k+1}\geq \liminf_{n\to\infty} \frac{b_{n-k}-b_{n-k-1}}{b_{n}}.
\end{eqnarray*}
\end{theo}

\begin{proof}
 We proceed as in the proof of Theorem \ref{satz2} using the fact that $d_{n}$ 
is increasing in place of the equivalent fact that $\frac{c_{m}}{c_{m+1}}$ is monotonically decreasing, 
which we deduced from the stronger assumptions in Theorem \ref{satz2}
 (where it was infered from the stronger assumptions in Theorem \ref{satz2}) up to equation (\ref{eq:qply}).
Instead of (\ref{eq:qply}), by our weaker assumption
$\frac{b_{n+1}}{b_{n}}>\kappa_{d}$ instead of $\frac{b_{n+1}}{b_{n}}>2>\kappa_{d}$, we obtain the weaker upper bound

\begin{equation*} 
 \omega_{j}^{2}(X_{r})\leq -\log_{X_{r}}\left(\frac{1}{c_{h_{r}+d+1-k}}\right)< -\log_{X_{r}}\left(\frac{c_{h_{r}+d+1-k}}{c_{h_{r}+d+2-k}}\right)\leq \omega_{k-d}^{1}(X_{r}),
\end{equation*}
which yields $\wp_{k-d}$ as an upper bound for $\omega_{j}$ instead of $\wp_{k-1}$.
We proceed analogously again till (\ref{eq:ronzor}), instead of which we obtain

\begin{equation*} 
 \omega_{j}^{2}(X_{r})\leq -\log_{X_{r}}\left(\frac{1}{c_{h_{r}+d+1-k}}\right)< -\log_{X_{r}}\left(\frac{c_{h_{r}+d+1-k}}{c_{h_{r}+d+2-k}}\right)\leq \omega_{k-d}^{1}(X_{r}).
\end{equation*}
This gives $\widehat{\wp}_{k-d}$ as an upper bound for $\widehat{\omega}_{j}$ instead of $\widehat{\wp}_{k-1}$.
The remainder of the proof is essentially the same as in Theorem \ref{satz2}.  
\end{proof}

\noindent Remark: We clearly have $\lim_{d\to\infty}\kappa_{d}=1$. On the other hand $1+\frac{1}{d}<\kappa_{d}<2$
which can esily be derived using the well known monotonic convergence of $(1+\frac{1}{n})^{n}$ to the Euler number $e\approx 2.71$.\\
\noindent Again, we easily deduce the following

\begin{coro}  \label{kor2}
Let the basic assumptions of Theorem \ref{satz2} and condition 1) from Theorem \ref{satz6} be satisfied. Let us further assume

\[
 \lim_{n\to\infty} \frac{b_{n+1}}{b_{n}}=C\geq \kappa_{d}.
\]
Then the first $(k-d)$ approximation constants are given by

\begin{eqnarray*}
 \omega \quad = &C-1&    \\
 \omega_{2} \quad = &\frac{C-1}{C}&=  \quad \widehat{\omega}                    \\
 \omega_{3} \quad = &\frac{C-1}{C^{2}}&=  \quad \widehat{\omega}_{2}               \\
 \vdots \qquad &\vdots& \quad \vdots                               \\
 \omega_{k-d} \quad = &\frac{C-1}{C^{k-d-1}}&=  \quad \widehat{\omega}_{k-d-1}    \\
 &\frac{C-1}{C^{k-d}}&=  \quad \widehat{\omega}_{k-d}.
\end{eqnarray*}
Further more we have the inequalities 

\begin{eqnarray*}
\omega_{k-d+1}&\geq& \frac{C-1}{C^{k-d}}                                    \\
\omega_{k-d+2} \quad \geq \quad \widehat{\omega}_{k-d+1}&\geq& \frac{C-1}{C^{k-d+1}}     \\
\vdots \qquad &\vdots& \qquad \vdots                                        \\
\omega_{k+1} \quad \geq \quad \widehat{\omega}_{k}&\geq& \frac{C-1}{C^{k}}               \\
 \widehat{\omega}_{k+1}&\geq& \frac{C-1}{C^{k+1}}.
\end{eqnarray*}
\end{coro}
\noindent Let us illustrate the results of Corollary \ref{korolar} and Corollary \ref{kor2} in the case $k=3,C=2$.
In the first plot we put
 \[
 \zeta_{1}^{\prime}= 2^{-1}+2^{-15}, \quad \zeta_{2}^{\prime}= 2^{-3}+2^{-31},\quad \zeta_{3}^{\prime}= 2^{-7}
\]
which (for numerical purposes) are the initial terms of

\[
 \zeta_{1}=\sum_{n\geq 0} 2^{-2^{3n+1}+1}, \quad \zeta_{2}=\sum_{n\geq 0}2^{-2^{3n+2}+1}, \quad \zeta_{3}=\sum_{n\geq 1}2^{-2^{3n}+1},
\]
which clearly satisfy the conditions of Corollary \ref{korolar} with $C=2$. 

\begin{figure}[h!]
 \scalebox{0.65}{\centering
 \includegraphics{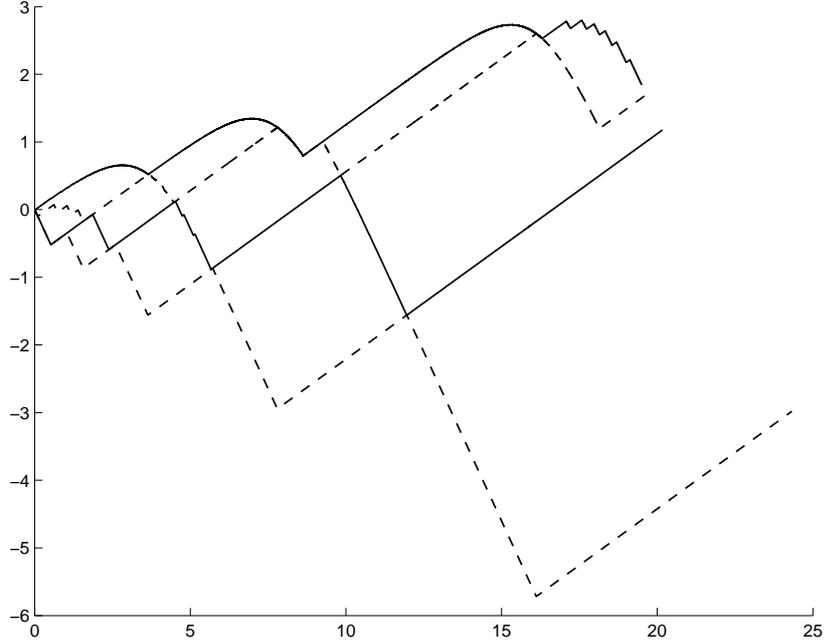}}
 \caption{k=3, C=2; illustrates Corollary 2.11}
\end{figure} 
\noindent Notice the special behaviour of $L_{k}=L_{3},L_{k+1}=L_{4}$ in comparison to the first $(k-1)=2$ functions which behave
as predicted in Corollary \ref{korolar}. \\
\noindent The assumptions of Corollary \ref{kor2} are weaker in the sense that either $C<2$ or the quotients $\frac{b_{n+1}}{b_{n}}$
converge to $C=2$ without being strictly larger than $2$ for every sufficiently large $n$. 
To illustrate this latter case we may put

\[
 \zeta_{1}^{\prime}=2^{-2}+2^{-9}, \quad \zeta_{2}^{\prime}= 2^{-3}+2^{-17}, \quad \zeta_{3}^{\prime}= 2^{-5}+2^{-33}.
\]
which are the initial terms of 

\[
 \zeta_{1}=\sum_{n\geq 0} 2^{-a_{3n+1}}, \quad \zeta_{2}= \sum_{n\geq 0} 2^{-a_{3n+2}},\quad \zeta_{3}=\sum_{n\geq 1} 2^{-a_{3n}}
\]
with $a_{1}=2$ and $a_{n+1}=2a_{n}-1$ for all $n\geq 1$,
which fufills the conditions of Corollary \ref{kor2} with $d=2$.
 Indeed, we will 
 see a different behaviour of $L_{2}$ compared to the previous
picture. Only $L_{k-d}=L_{1}$ has the predicted shape.

\begin{figure}[h!]
 \scalebox{0.65}{\centering
 \includegraphics{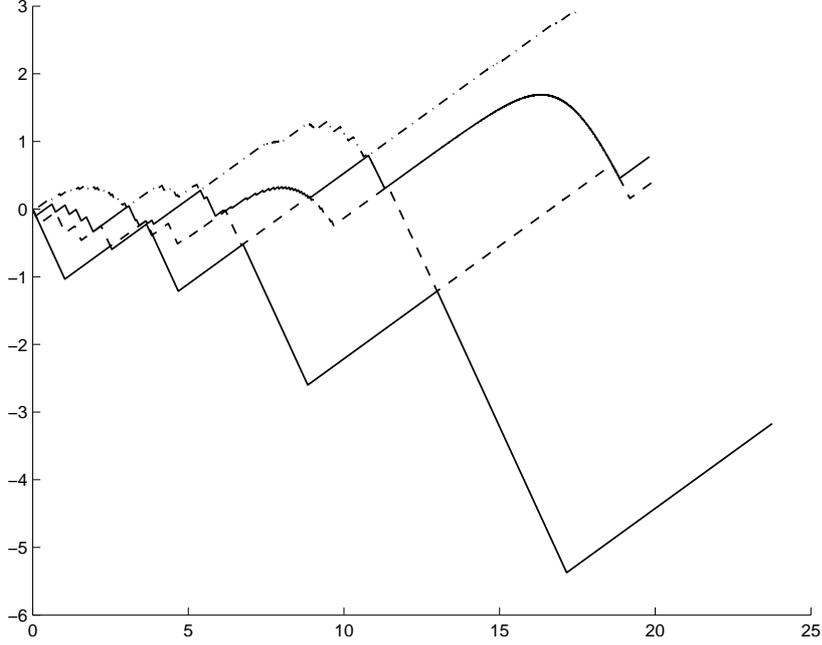}}
 \caption{k=3, C=2; illustrates Corollary 2.13}
\end{figure} 
\noindent We can apply Corollary \ref{kor2} to construct many more cases of Schmidt's conjecture explicitely.
For simplicity of the proof we first deduce another easy Corollary from Corollary \ref{kor2}.

\begin{coro} \label{coroll}
 For $k\geq 2$ let $1\leq d\leq k-1$ be an arbitrary integer.
For any $C>\kappa_{d}$ there exists a sequence of positive
integers $(b_{n})_{n\geq 1}$ such that $\frac{b_{n+1}}{b_{n}}>C$ for all $n$ and
$\lim_{n\to\infty}\frac{b_{n+1}}{b_{n}}=C$.
If $(a_{n,j})_{n\geq 1}$ for $1\leq j\leq k$ are $k$ sequences satisfying (\ref{eq:frure}) such that
$(b_{n})_{n\geq 1}$ is their ordered mixed sequence, then for $\zeta_{j}=\sum_{n\geq 1}2^{-a_{n,j}}$ the result
of Corollary \ref{kor2} is valid. \\
Furthermore we can choose the sequence $(b_{n})_{n\geq 1}$ such that $1,\zeta_{1},\ldots ,\zeta_{k}$
are $\mathbb{Q}$-linearly independent.
\end{coro}

\begin{proof}
 The sequence $(b_{n})_{n\geq 1}$ defined by $b_{1}=S$ and $b_{n+1}=\left\lceil Cb_{n}\right\rceil$
with $S$ sufficiently large that $b_{2}>b_{1}$ clearly satisfies the stated properties.
 Putting $q_{n,j}=2^{a_{n,j}}$ we see that the assumptions of Corollary \ref{kor2} are satisfied,
since it clearly makes no difference if we take the logarithm to base $2$ instead of $e$ as the quotients 
$\frac{b_{n+1}}{b_{n}}$ don't change. By a variation of $(b_{n})_{n\geq 1}$ as in the proof of 
Corollary \ref{prop1} we can guarantee the linear independence.
\end{proof}

\begin{coro}
With the notation of Corollary \ref{coroll} there exists a constant $R(k)$ such that
the following holds:
\begin{itemize}
 \item $R(k)> \frac{k}{\log(k)}$ for $k$ sufficiently large.\\
\noindent \item For fixed $3\leq T\leq R(k)$ there is some $C_{0}=C_{0}(T)$, such that
there exists a sequence $(b_{n})_{n\geq 1}$ of positive integers 
satisfying $\lim_{n\to\infty}\frac{b_{n+1}}{b_{n}}=C_{0}$ such that the corresponding vector
$(\zeta_{1},\ldots ,\zeta_{k})$ constructed via Corollary \ref{coroll} with $C=C_{0}$
has approximation constants that satisfy  

\begin{equation} \label{eq:frrranr}
 \overline{\psi}_{T-2}<0 \quad \text{and} \quad \underline{\psi}_{T}>0.
\end{equation} 
A possible choice of $R(k)$ is $R(k):=k-1+(k-2)\frac{\log\left(k^{\frac{1}{k-2}}-1\right)}{\log(k)}$. 
This provides {\em explicit examples} for Schmidt's Conjecture.

\end{itemize}

\end{coro}

\begin{proof}
Let $k$ be an arbitrary but fixed integer. In view of (\ref{eq:2})  for a given $3\leq T\leq R(k)$
we need to find
$C_{0}=C_{0}(T)$ such that a vector $(\zeta_{1},\zeta_{2},\ldots ,\zeta_{k})$ that arises from Corollary \ref{coroll}
with $C=C_{0}$ satisfies $\widehat{\omega}_{T-2}>\frac{1}{k}>\omega_{T}$ 
 to obtain (\ref{eq:frrranr}). 
We will implicitly identify such a vector $(\zeta_{1},\zeta_{2},\ldots ,\zeta_{k})$
 with the resulting value $C$ from the limit of the quotients $\frac{b_{n+1}}{b_{n}}$ in Corollary \ref{coroll}.
This is well defined as the approximation constants we consider don't depend on the choice of the exact vector
but depend only on $C$ by Corollary \ref{coroll}. As was shown there the set of such vectors $(\zeta_{1},\zeta_{2},\ldots ,\zeta_{k})$ is
 nonempty for every $C>\kappa_{d}>1$ and we can assume $(\zeta_{1},\ldots ,\zeta_{k})$ 
together with $1$ to be $\mathbb{Q}$ linearly independent. \\
\noindent For any positive integer $u$ define the function $\Psi_{u}(x)=\frac{x-1}{x^{u}}$.
Each function $\Psi_{.}$ is easily seen to be
continuous and $\Psi_{u}$ increases on $[1,\frac{u}{u-1}]$ and decreases on $[\frac{u}{u-1},\infty)$ with limit $0$ as $x\to\infty$. \\
\noindent We use the notation of Theorem \ref{satz6}, in particular $\kappa_{d}$ is the largest real root of $P_{d}(x)=x^{d}-x^{d-1}-1$.
 We first prove that we can choose $C_{0}=C_{0}(T)$ such that (\ref{eq:frrranr}) holds for 
a given $T$ that has the property  

\begin{equation} \label{eq:turmm}
 \Psi_{T-1}(\kappa_{k-T})>\frac{1}{k}
\end{equation}
with the constructions of Corollary \ref{coroll} and the particular choice $C=C_{0}$.\\ 
Put $u=T-1$. If (\ref{eq:turmm}) is valid, the facts about the functions $\Psi_{u}$ show that there is $x>\kappa_{k-T}$
such that $\Psi_{T-1}(x)=\frac{1}{k}$ and $\Psi_{T-1}$ already decreases at $x$. 
Furthermore there is an interval $C_{0}\in{(x,x+\delta)}$ such that $\Psi_{T-1}(C_{0})<\frac{1}{k}<\Psi_{T-2}(C_{0})$. 
Since $C_{0}>x\geq \kappa_{k-T}$ we can apply Corollary \ref{coroll} with $d:=k-T, C:=C_{0}$  and obtain 
$\omega_{T}=\Psi_{T-1}(C_{0})<\frac{1}{k}<C_{0}\frac{1}{k}=C_{0}\Psi_{T-1}(C_{0})=\Psi_{T-2}(C_{0})<\widehat{\omega}_{T-2}$, as intended. \\

\noindent Now assume $k$ is fixed, $1\leq T\leq k+1$ and that for $C_{0}:=\kappa_{k-T}$ 

\begin{equation}  \label{eq:skkiforn}
 \frac{C_{0}-1}{C_{0}^{k-1}-C_{0}^{k-2}}>\frac{1}{k}  
\end{equation}
holds. By definition of $C_{0}=\kappa_{k-T}$ we have $C_{0}^{T-1}=C_{0}^{k-1}-C_{0}^{k-2}$ and (\ref{eq:skkiforn}) further implies

\begin{equation*} 
 \psi_{T-1}(C_{0})=\frac{C_{0}-1}{C_{0}^{T-1}}= \frac{C_{0}-1}{C_{0}^{k-1}-C_{0}^{k-2}}>\frac{1}{k}
\end{equation*}
for such $T$, i.e. (\ref{eq:turmm}). Combining what we have shown so far, (\ref{eq:frrranr}) can
be obtained for $T$ and $C_{0}(T)=\kappa_{k-T}$, provided (\ref{eq:skkiforn}) holds.
We now show that (\ref{eq:skkiforn}) is true for $3\leq T\leq R(k)$. \\
In view of $C_{0}>1$, inequality (\ref{eq:skkiforn}) is equivalent to $\kappa_{k-T}=C_{0}\leq k^{\frac{1}{k-2}}$. Since
$P_{k-T}$, whose largest root is $\kappa_{k-T}>1$, increases on the interval $[1,\infty)$,
 this is equivalent to $P_{k-T}(k^{\frac{1}{k-2}})\geq 0$, i.e. 

\[
 k^{\frac{k-T}{k-2}}-k^{\frac{k-T-1}{k-2}}-1>0.
\]
Basic rearrangements show this is equivalent to

\[
 T\leq k-1+(k-2)\frac{\log\left(k^{\frac{1}{k-2}}-1\right)}{\log(k)}=:R(k).
\]
To finish up the proof we are left to show that $R(k)>\frac{k}{\log(k)}$ holds for $k$ sufficiently large.
To see this we claim that for sufficiently large $k$ we have

\[
R(k)> \frac{k}{2}\left[1+\frac{\log(k^{\frac{1}{k-2}}-1)}{\log(k)}\right]> \frac{k}{2}\left[1+\frac{\log(k^{\frac{1}{k}}-1)}{\log(k)}\right]>\frac{k}{\log(k)}. 
\]
For $k\geq 4$ only the right most inequality is non trivial. It is equivalent to $k\cdot (k^{\frac{1}{k}}-1)>e^{2}$ or $k^{\frac{1}{k}}>\frac{e^{2}}{k}+1$.
However, this 
is easily seen to be true by $k^{\frac{1}{k}}>1+\frac{M}{k}$ for $k\geq k_{0}(M)$ for arbitrary $M$,
 which follows from the monotonic convergence of $\lim_{k\to\infty}(1+\frac{M}{k})^{k}=e^{M}$ for every $M$.
Putting $M>e^{2}$ arbitrary proves the claim. 
\end{proof}
\noindent Remark: We have $\lim_{k\to\infty}\frac{R(k)}{k}=0$, but the convergence is very slow, in particular we have seen
it is slower than $\frac{1}{\log(k)}$. 

\section{The case $\underline{\psi}_{j+1}=\overline{\psi}_{j}$}

\noindent In this chaper we want to treat the case of $\zeta_{1},\zeta_{2},\ldots ,\zeta_{k}$ that define functions $\psi_{j}$ with
the property 

\begin{equation} \label{eq:rindvieg}
 \underline{\psi}_{j+1}=\overline{\psi}_{j},\qquad 1\leq j\leq k.
\end{equation}
We will exclude
the generic case where all approximation constants $\underline{\psi}_{j},\overline{\psi}_{j}$ are
zero, and call the cases where $\underline{\psi}_{1}=-1$ the {\em degenerate cases}.
Note, that the numbers Corollary \ref{korolar} deals with lead to functions $\psi_{j}$ that satisfy the 
equalities of (\ref{eq:rindvieg}) for $1\leq j\leq k-2$,
but not for $j\in{\{k-1,k\}}$ in the case $C<\infty$ (see the bounds given in the remark following Corollary \ref{korolar}).
A {\em special case} of (\ref{eq:rindvieg}), for which an idealized picture is shown below, is of particular interest.

\begin{figure}[h!]
 \psfrag{a}{$q_{s}$}
 \psfrag{b}{$q_{s+1}$}
 \psfrag{c}{$q_{s+2}$}
 \psfrag{d}{$\underline{\psi}_{1}$}
 \psfrag{e}{$\overline{\psi}_{1}=\underline{\psi}_{2}$}
 \psfrag{f}{$\overline{\psi}_{2}=\underline{\psi}_{3}$}
 \psfrag{g}{$\overline{\psi}_{3}$}
 \psfrag{h}{$p$}
 \psfrag{i}{$q_{s}$}
 \psfrag{j}{$q_{s+1}$}
 \psfrag{k}{$q_{s+2}$}
 \psfrag{r}{$x_{i}$}
 \psfrag{s}{$x_{i+1}$}
 \psfrag{t}{$x_{i+2}$}
 \psfrag{ell}{$q$}
 \psfrag{m}{$L_{1}$}
 \psfrag{n}{$L_{2}$}
 \psfrag{o}{$L_{3}$}
 \scalebox{0.65}{\centering
 \includegraphics{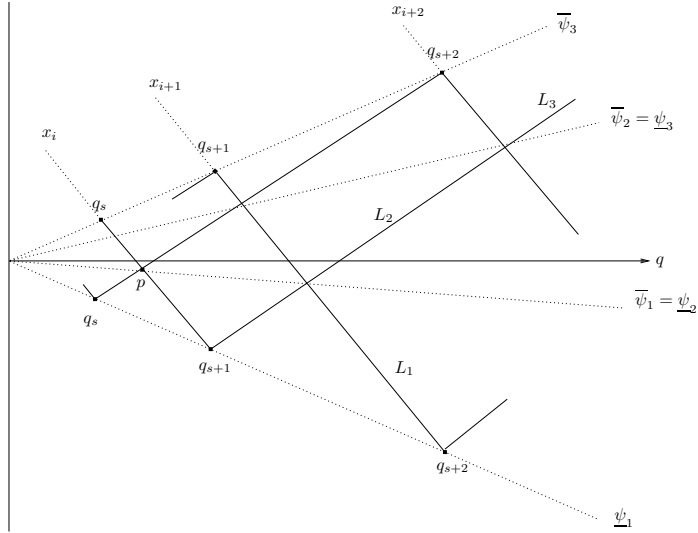}}
 \caption{$\underline{\psi}_{j+1}=\overline{\psi}_{j}$, special case}
\end{figure} 

\noindent If $(x^{(i)})_{i\geq 1}$ denotes the sequence of the first coordinates of approximation vectors 
$(x,y_{1},\ldots ,y_{k})$, then we consider the special case where there is a sequence of $(x^{(i)})_{i\geq 1}$
 with (ideally) constant value of $\frac{\log(x^{(i+1)})}{\log(x^{(i)})}$ such that 

\[
 \max_{1\leq j\leq k}\vert \zeta_{j}x^{(i)}-y_{j}\vert \thicksim (x^{(i)})^{-\omega}.
\]
This sequence causes the second minimum to attain the value $\omega_{2}$, and so on. Thus in particular we have 

\begin{eqnarray}
 \lim_{i\to\infty} \frac{\log\left(x^{(i+1)}\right)}{\log\left(x^{(i)}\right)}&=& \frac{\omega_{j}}{\omega_{j+1}}, \qquad 1\leq j\leq k   \label{eq:pattricia}  \\
 \lim_{i\to\infty} -\frac{\log\left(\max_{1\leq t\leq k} \left\vert x^{(i)}\zeta_{t}-y_{t}^{(i)}\right\vert\right)}{\log\left(x^{(i)}\right)}&=& \omega. \label{eq:trapizia} 
\end{eqnarray}
\noindent Let the equistence of a sequence $x^{(i)})_{i\geq 1}$ such that (\ref{eq:pattricia}),(\ref{eq:trapizia})
holds and additionally for $i\geq i_{0}$ 
every $k+1$ consecutive approximation vectors $(x^{(j)},y_{1}^{(j)},\ldots,y_{k}^{(j)})$
belonging to $x^{(i)},x^{(i+1)},\ldots ,x^{(i+k)}$ (ie $j\in{\{i,i+1,\ldots ,i+k\}}$) are linearly independent 
be our definition of the {\em special case} mentioned above.\\
\noindent Roy shows, that numbers he defines as {\em extremal numbers} $\zeta$ in the introduction of \cite{4} satisfy the property of the special case of (\ref{eq:rindvieg})           
for $k=2$ and $\zeta_{1}=\zeta, \zeta_{2}=\zeta^{2}$ and yield $\omega_{j}=\gamma^{j-1}$ for $1\leq j\leq 3$
 and $\widehat{\omega}_{3}=\gamma^{3}$ with $\gamma:=\frac{\sqrt{5}-1}{2}$. We are interested in other particular cases
of the special case of (\ref{eq:rindvieg}). \\
It follows in general by (\ref{eq:pattricia}),(\ref{eq:trapizia}) that all the values $\underline{\psi}_{.},\overline{\psi}_{.}$ are determined by the value
 $\underline{\psi}_{1}$ (or equivalently $\omega$). 
This holds for the degenerate case in particular. However, we will show in Theorem \ref{theo9} that this phenomenon holds 
for {\em all} $(\zeta_{1},\zeta_{2},\ldots ,\zeta_{k})$ in the degenerate case of (\ref{eq:rindvieg}). 
By virtue of Corollary \ref{korolar} we can easily provide concrete examples for the degenerate case.
Before we do so, for the sake of completeness we give a general result about the degenerate case of (\ref{eq:rindvieg}).

\begin{prop} \label{kokorp}
Assume the approximation functions arising from $\zeta_{1},\ldots ,\zeta_{k}$ satisfy \\
 $\underline{\psi}_{1}=-1$ and (\ref{eq:rindvieg}). Then they already satisfy

\begin{eqnarray} 
 \underline{\psi}_{1}&=& -1  \label{eq:obdasgeht}  \\
  \overline{\psi}_{1}&=& \frac{1-k}{2k} =\underline{\psi}_{2}  \label{eq:dasauch}  \\
  \overline{\psi}_{j}&=& \frac{1}{k}= \underline{\psi}_{j+1}, \qquad 2\leq j\leq k,     \label{eq:unddasauch}  \\
 \overline{\psi}_{k+1}&=& \frac{1}{k},          \label{eq:unddies}
\end{eqnarray}

and hence in particular fall under the special case.
\end{prop}

\begin{proof}
 First note that in general if $\underline{\psi}_{1}=-1$ we have $\overline{\psi}_{j}=\frac{1}{k}$ for $2\leq j\leq k+1$
by means of (\ref{eq:zokz}), see the proof of Theorem \ref{satz1}. Consequently by (\ref{eq:rindvieg}) we have
(\ref{eq:unddasauch}) and (\ref{eq:unddies}). \\
\noindent For (\ref{eq:dasauch}) note first that $\frac{1-k}{2k}$ is always a lower bound as established in (\ref{eq:1000alk}),(\ref{eq:1001alk}).
So by (\ref{eq:rmuo}) it suffices to prove $\overline{\psi}_{1}\leq \frac{1-k}{2k}$.\\
\noindent Suppose we had $\overline{\psi}_{1}>\frac{1-k}{2k}$. This means for some sequence $(q_{n})_{n\geq 1}$ tending to infinity 
we have $\psi_{1}(q_{n})>V$ for some $V>\frac{1-k}{2k}$. Putting $V_{0}:=2(V-\frac{1-k}{2k})>0$ and using $\psi_{2}(q_{n})\geq \psi_{1}(q_{n})$ we have

\[
 \lim_{n\to\infty}\sum_{j=1}^{k+1} \psi_{j}(q_{n})> 2V+(k-1)\left[\frac{1}{k}-\frac{V_{0}}{2(k-1)}\right]=\frac{V_{0}}{2}>0, 
\]
a contradiction to (\ref{eq:zokz}) since $\lim_{n\to\infty} q_{n}=\infty$.
\end{proof}

\begin{theo} \label{theo9}
For any $k\geq 2$ there exist real numbers $\zeta_{1},\zeta_{2},\ldots,\zeta_{k}$ as in Corollary \ref{korolar}
with $C=\infty$ together with $1$
 linearly independent over $\mathbb{Q}$ that satisfy the degenerate case of (\ref{eq:rindvieg}) and hence
(\ref{eq:obdasgeht})-(\ref{eq:unddies}) by Proposition \ref{kokorp}.

\end{theo}

\begin{proof}
 Using (\ref{eq:2}) we obtain the equivalent system

\begin{eqnarray}
 \omega&=&\infty    \label{eq:90alk} \\
 \widehat{\omega}&=&1 =\;\; \omega_{2}  \label{eq:91alk}  \\
 \widehat{\omega}_{j}&=& 0=\;\; \omega_{j+1}, \qquad 2\leq j\leq k   \label{eq:92alk}     \\
 \widehat{\omega}_{k+1}&=& 0.         \label{eq:93alk}
\end{eqnarray}
In the case $k\geq 3$, we can just apply Corollary \ref{korolar} with $C=\infty$ because then we obviously have

\[
 \lim_{C\to\infty} \frac{C-1}{C}=1,\qquad \lim_{C\to\infty} \frac{C-1}{C^{j}}= 0,\quad j\geq 2,
\]
which gives (\ref{eq:90alk})-(\ref{eq:93alk}). However, in the case $k=2$ and $C=\infty$ we can also apply 
Corollary \ref{korolar} with a slightly more sophisticated argumentation. Of course we directly infer (\ref{eq:90alk})
and the left equation in (\ref{eq:91alk}) follows as for $k\geq 3$. From (\ref{eq:90alk}) we can immediately
 deduce $\widehat{\omega}_{j}=0$ for $j=2,3$ as in the proof of Theorem \ref{satz1}, which is a rephrasing of the left hand side 
of (\ref{eq:92alk}) and (\ref{eq:93alk}). By the left equation in (\ref{eq:91alk}) and Jarnik's identity
$\omega_{3}+\widehat{\omega}=1$ (see comments between the end of proof of Theorem \ref{satz1} and Corollary \ref{koroo}) we get $\omega_{3}=0$,
 i.e. the right hand side of (\ref{eq:92alk}). For the missing right hand equation of (\ref{eq:91alk}) note that
 $\omega_{2}\geq \widehat{\omega}$ is always true by (\ref{eq:rmuo}) and on the other hand
$\omega_{2}\leq 1$ by (\ref{eq:2222neu}), so by the left hand
equality in (\ref{eq:91alk}) we infer the right hand equality of (\ref{eq:91alk}). 
\end{proof}
\noindent This allows to show that the bounds in (\ref{eq:2000neu})-(\ref{eq:2004neu}) are best possible if considered independently
by using only three types of vectors $(\zeta_{1},\zeta_{2},\ldots ,\zeta_{k})$ depending on the dimension $k$.
These types are:
\begin{itemize}
 \item a set of together with $1$ $\mathbb{Q}$-linearly independent {\em algebraic} numbers $\zeta_{1},\zeta_{2},\ldots ,\zeta_{k}$ (leading to the generic case)  \\
 \item $\zeta_{1},\zeta_{2},\ldots ,\zeta_{k}$ as in Corollary \ref{korolar} with $C=\infty$, for example $\zeta_{j}=\sum_{n\geq 1} \frac{1}{(nk+j)!}$ for $1\leq j\leq k$       \\
 \item $\zeta_{1},\zeta_{2},\ldots ,\zeta_{k}$ as in Corollary \ref{korol1}
\end{itemize}

\begin{coro}  \label{kor_11}
The bounds (\ref{eq:2000neu})-(\ref{eq:2004neu}) are all (each for itself) 
optimal among $(\zeta_{1},\zeta_{2},\ldots ,\zeta_{k})$ that are 
$\mathbb{Q}$-linearly independent together with $1$.
\end{coro}

\begin{proof}
In Corollary \ref{korol1} we have seen, that the upper bounds in (\ref{eq:2000neu}),(\ref{eq:2222neu}),
(\ref{eq:2002neu}) as well as the lower bounds in 
(\ref{eq:2001neu}),(\ref{eq:2003neu}),(\ref{eq:2004neu}) cannot be improved.\\
In Theorem \ref{theo9} we've just seen, that the left hand side of (\ref{eq:2002neu})
and the right hand side of (\ref{eq:2001neu}) are optimal.\\
However, all the other bounds are $1/k$ and it is well known that all constants
 $\omega_{j},\widehat{\omega}_{j}$ are equal to $1/k$ in the generic case. To give concrete examples, an implication of
Schmidt's subspace theorem says, that for all $\mathbb{Q}$-linearly independent algebraic numbers 
all approximation constants take the value $1/k$ (which follows already from $\omega=1/k$ by (\ref{eq:2}) and (\ref{eq:zokz})). 
So the lower bounds of (\ref{eq:2000neu}),(\ref{eq:2222neu}) such as the upper bounds of (\ref{eq:2003neu}),
(\ref{eq:2004neu}) cannot be improved either, and the list is complete.
\end{proof}
\noindent Let $\zeta_{1},\zeta_{2},\ldots ,\zeta_{k}$ be real numbers that lead to a special case of (\ref{eq:rindvieg}), i.e. 
(\ref{eq:pattricia}),(\ref{eq:trapizia}) hold.
It follows directly from (\ref{eq:pattricia}),(\ref{eq:trapizia}) that all the constants $\omega_{j},\widehat{\omega}_{j}$ 
only depend on $\omega$. It is easy to check that more precisely we have

\begin{eqnarray} 
 \frac{(1+\omega)^{k+1}}{\omega}&=& \frac{(1+\widehat{\omega}_{k+1})^{k+1}}{\widehat{\omega}_{k+1}}  \label{eq:hunggar}  \\
\omega_{j}&=& \omega^{1-\frac{j-1}{k+1}}\widehat{\omega}_{k+1}^{\frac{j-1}{k+1}}, \qquad 1\leq j\leq k+1.   \label{eq:turnfisch}
\end{eqnarray}
Using this we now prove a lower bound for $\widehat{\omega}$ in dependence of $\omega$.

\begin{prop} \label{wasqob}
 In the special case of (\ref{eq:rindvieg}) for $k\geq 2$ we have

\[
 \frac{\omega}{\omega+1}< \widehat{\omega}\leq 1.
\]

\end{prop}

\begin{proof}
 The right hand side inequality is just (\ref{eq:2001neu}).\\
\noindent Suppose for some $k\geq 2$ we had 
$\widehat{\omega}\leq \frac{\omega}{\omega+1}$. Putting $j=2$ in (\ref{eq:turnfisch}) (note $\omega_{2}=\widehat{\omega}$ by definition)
we have

\begin{equation} \label{eq:roott}
 \widehat{\omega}_{k+1}\leq \left[\left(\frac{\omega}{\omega+1}\right)\omega^{-\frac{k}{k+1}}\right]^{k+1}=\frac{\omega}{(\omega+1)^{k+1}}.
\end{equation}
Denote 

\[
 f_{k}(x):= \frac{(x+1)^{k+1}}{x}, \qquad k\geq 1.
\]
Differentiating shows that $f_{k}$ decreases on $x\in{(0,\frac{1}{k})}$ and increases on $x\in{(\frac{1}{k},\infty)}$,
so its global mimimum on $(0,\infty)$ is at $x=\frac{1}{k}$. Combining this with 
$\widehat{\omega}_{k+1}<\frac{1}{k}$, (\ref{eq:hunggar}) and (\ref{eq:roott}) we obtain

\[
 f_{k}(\omega)= f_{k}(\widehat{\omega}_{k+1})\geq f_{k}\left(\frac{1}{f_{k}(\omega)}\right), \qquad k\geq 1.
\]
Putting $z:=\frac{1}{f_{k}(\omega)}$ this gives $\frac{1}{z}\geq f_{k}(z)$, which is false, as $\frac{1}{z}$ is an
expression in the binomial expansion of $f_{k}(z)$.
\end{proof}

\noindent Remarks: 1) One can proof that for $k\geq 2$ we have $\lim_{\omega\to\infty}\omega+1-\frac{\omega}{\widehat{\omega}}=0$.\\
Observe, that in Corollary \ref{kor2} with arbitrary $C$ we always have $\widehat{\omega}=\frac{\omega}{\omega+1}$. Proposition \ref{wasqob} 
shows that given $\omega$ the resulting special case of (\ref{eq:rindvieg}) leads to a larger value of $\widehat{\omega}$.
It may be conjectured that among all $\zeta_{1},\zeta_{2},\ldots ,\zeta_{k}$
linearly independent together with $1$ with prescribed $\omega=\omega_{0}$, the quantity $\widehat{\omega}$ is maximised for the special case of (\ref{eq:rindvieg})
with the value $\omega=\omega_{0}$. \\
\noindent 2) Observe that the inequality $\frac{\widehat{\omega}^{2}}{1-\widehat{\omega}}\leq \omega$ always holds as established by Jarnik, see Theorem 1 page 331 in \cite{2}.                 
 So together with Proposition \ref{wasqob} in the special case of (\ref{eq:rindvieg}) we have 

\[
 \frac{\widehat{\omega}^{2}}{1-\widehat{\omega}}\leq \omega \leq \frac{\widehat{\omega}}{1-\widehat{\omega}}.
\]

\vspace{2cm}

\noindent The author thanks L. Summerer for help in translation and W. Schmidt for an idea to 
generalize my initial versions of Theorems \ref{satz1}, \ref{satz2}, \ref{satz6} which simplified the proofs and made them more concise.

\end{document}